\documentclass[aihp,authoryear]{imsart}
\usepackage{amssymb,amsmath,amsthm,natbib,graphics,color}
\usepackage{dsfont}
\RequirePackage{natbib}
\usepackage{graphicx}
\RequirePackage[OT1]{fontenc}
\RequirePackage[colorlinks,citecolor=blue,urlcolor=blue]{hyperref}
\newcommand{\lebesgue}{\ensuremath{\lambda\hspace{-1.1ex}\lambda}}
\newcommand*{\bigtimes}{\mathop{\raisebox{-.5ex}{\hbox{\huge{$\times$}}}}}

\allowdisplaybreaks
\startlocaldefs
\numberwithin{equation}{section}
\theoremstyle{plain}
\newtheorem{thm}{Theorem}[section]
\newtheorem{theorem}[thm]{Theorem}
\newtheorem{proposition}[thm]{Proposition}
\newtheorem{lemma}[thm]{Lemma}
\theoremstyle{remark}
\newtheorem{remark}[thm]{Remark}
\theoremstyle{definition}
\newtheorem{definition}[thm]{Definition}
\theoremstyle{definition}
\newtheorem{lemdef}[thm]{Lemma and Definition}
\theoremstyle{theorem}
\newtheorem{example}[thm]{Example}
\endlocaldefs
\begin{document}
\begin{frontmatter}
\title{Joint exceedances of random products}
\runtitle{Joint exceedances of random products}

\begin{aug}
\author{\fnms{Anja} \snm{Jan\ss en}\thanksref{a}\ead[label=e1]{anja.janssen@math.uni-hamburg.de}}
\and
\author{\fnms{Holger} \snm{Drees}\thanksref{a}\ead[label=e2]{holger.drees@math.uni-hamburg.de}}
\thankstext{a}{Both authors were supported by the German Research Foundation DFG, grant no JA 2160/1.}

\address[a]{University of Hamburg, Department of Mathematics, SPST, Bundesstr.~55, D-20146 Hamburg, Germany\\
\printead{e1,e2}}

\runauthor{A. Jan\ss en and H. Drees}

\affiliation{University of Hamburg}

\end{aug}

\begin{abstract}
We analyze the joint extremal behavior of $n$ random products of the form $\prod_{j=1}^m X_j^{a_{ij}}, 1 \leq i \leq n,$ for non-negative, independent regularly varying random variables $X_1, \ldots, X_m$ and general coefficients $a_{ij} \in \mathbb{R}$. Products of this form appear for example if one observes a linear time series with gamma type innovations at $n$ points in time. We combine arguments of linear optimization and a generalized concept of regular variation on cones to show that the asymptotic behavior of joint exceedance probabilities of these products is determined by the solution of a linear program related to the matrix $\mathbf{A}=(a_{ij})$.
\end{abstract}


\end{frontmatter}
\section{Introduction}

The tail behavior of products of powers of heavy-tailed positive random variables is of crucial importance in many applications, particularly in finance, but e.g.\ in network modeling too. In stochastic volatility time series, the log-volatilities are usually modelled as linear time series
 \begin{equation*}
 \log \sigma_t=\sum_{i=0}^\infty \alpha_i \eta_{t-i}, \;\;\; t \in \mathbb{Z},
 \end{equation*}
 If the innovations $\eta_i, i \in \mathbb{Z},$ have an exponential or gamma type tail, then the volatility $\sigma_t$ at time $t$ is a product of powers of the regularly varying random variables $X_i:=e^{\eta_i}, i \in \mathbb{Z}$, with exponents depending on $t$. To assess the risk of a volatile market at different time points $t_1,\ldots, t_n$, one thus has to analyze probabilities of the type
 $ P(\prod_{j=1}^\infty X_j^{a_{ij}}>x,\;1 \leq i \leq n)$ for suitable exponents $a_{ij}$. Using a new Breiman type result, it was shown in \cite{JaDr15} that these probabilities also determine the risk of jointly large losses over different periods.

 Similarly, in a credit risk model for $n$ risks with $k$ independent factors $Z_1,\ldots, Z_k$, the $i$-th risk is often modeled as a multiple of $\exp(\sum_{j=1}^k a_{ij} Z_j+Y_i)$ with $Y_i, 1 \leq i \leq n,$ denoting the idiosyncratic part (cf.\ \cite{EmHaMi14}). If the $Z_j, 1 \leq j \leq k,$ and $Y_i, 1 \leq i \leq n,$ have an exponential or gamma type tail, the analysis of the joint tail risk again leads to probabilities of the above type.

 In network modeling, both transmission durations $L$ and rates $R$ arising from one source may be modeled by regularly varying random variables with different indices $\alpha_L$ and $\alpha_R$ (see, e.g., \cite{MauReRo02}). The total volume of traffic from one source can then be expressed as $X_L^{\alpha_L} X_R^{\alpha_R}$ for random variables $X_L,X_R$ which are regularly varying with index $-1$. If one wants to determine the probability that different sources contribute large volumes in the same period, then again probabilities of the above type arise. Moreover, one may introduce dependencies between $X_L$ and $X_R$ for the same or for different sources by modeling them as products of (partially) identical factors with different exponents.

 As the example of log-volatilities demonstrates, the analysis of the joint tail behavior of power products is equivalent to the corresponding analysis for linear combinations of random variables with exponential type tails such that exponentials of these random variables are regularly varying. Hence the results given below allow a tail analysis in such settings too.

 Motivated by these examples, we analyze the asymptotic behavior of the probability of joint exceedances of power products, i.e. of
 \begin{equation}\label{Eq:powerproductsmot}
 P\left(\prod_{j=1}^m X_j^{a_{ij}}>c_i x, 1 \leq i \leq n\right), \;\;\; c_i>0, 1 \leq i \leq n,
 \end{equation}
 where $X_j$ are independent, non-negative regularly varying random variables and $a_{ij}$ are coefficients which may be negative. We restrict the analysis to the case of a finite number $m$ of factors, but extensions to an infinite number of factors, using a Breiman-type argument, are possible. In \cite{JaDr15}, such probabilities were investigated  under the restrictive assumption that no coefficient is negative. It was shown there, that the probabilities behaved asymptotically like a multiple of $\prod_{i=1}^m P(X_i>x^{\kappa_i})$, where $\boldsymbol{\kappa}=(\kappa_1, \ldots, \kappa_m)$ is the solution to a linear program determined by the matrix $\mathbf{A}=(a_{ij})$. While the restriction to positive coefficients $a_{ij}$ seems acceptable e.g.\ for multi-factor models, it is quite severe for log-volatility time series. Moreover, essentially only the case $n=2$ was considered in \cite{JaDr15} and the techniques employed do not easily generalize to higher dimensions, which limits the applicability of the established
results further.

 Using a recently introduced abstract concept of regular variation on cones based on the notion of $\mathbb{M}$-convergence (see \cite{HuLi06}, \cite{LiReRo14}), we can avoid all these drawbacks. To this end, we first introduce a non-standard form of regular variation on the cone $(0, \infty)^m$ for the random vector $(X_1, \ldots, X_m)$, from which one may conclude the asymptotics of probabilities of the type $P((X_1/x^{\kappa_1}, \ldots, X_m/x^{\kappa_m})\in B)$ for suitable coefficients $\kappa_1, \ldots, \kappa_m$ and sets $ B \subset(0, \infty)^m$ that are bounded away from the boundary of the cone. Unfortunately, in general the sets $M=\{x\in\mathbb{R}^m: \prod_{j=1}^m x_j^{a_{ij}}>c_i, 1 \leq i \leq n\}$ pertaining to the probabilities \eqref{Eq:powerproductsmot} are not of this type. Hence, quite involved arguments are needed to prove that the parts of $M$ close to the boundary of the cone are asymptotically negligible. To this end, auxiliary results are proved in Section \ref{Sec:auxresults} which
are
of interest on
their own. In particular, Proposition \ref{Lem:momentsconverge} can be seen as a multivariate version of the direct half of Karamata's Theorem, cf.\ Remark \ref{Rem:Karamata}.

The outline of this paper is as follows: In Section \ref{Sec:background} an abstract notion of regular variation on cones is briefly recalled, with a view towards the later application of this concept. Our main results are stated in Section \ref{Sec:mainsec}, with Theorem \ref{Th:maintheorem} being the central conclusion. Proofs of the results are given in Section \ref{Sec:proof} while some auxiliary results needed in the proofs are gathered in Section \ref{Sec:auxresults}.

\subsection*{Notations and conventions}
We write bold letters for vectors, i.e. $\mathbf{x}$ is short for $(x_1, \ldots, x_n) \in \mathbb{R}^n$ if it is clear that $\mathbf{x}$ is of dimension $n$. The $i$-th component of $\mathbf{x}$ is denoted by $x_i$. We write $\mathbf{0}$ and $\mathbf{1}$ for a (column) vector of suitable dimension which consists of only zeros or only ones. Inequalities for vectors are meant to hold componentwise. We denote the complement of a set $A$ by $A^c$ and its boundary by $\partial A$.
For $\mathbf{x} \in \mathbb{R}^n$ and $A \subset \mathbb{R}^n$ we set $d(\mathbf{x},A)=\inf_{\mathbf{a} \in A}\|\mathbf{x}-\mathbf{a}\|$, where $\| \cdot \|$ denotes the Euclidean norm. Similarly, we set $d(A,B)=\inf_{\mathbf{a} \in A, \mathbf{b} \in B}\|\mathbf{a}-\mathbf{b}\|$. For $A \subset \mathbb{R}^n$ and $r>0$ set $A^r:=\{\mathbf{x} \in \mathbb{R}^n: d(\mathbf{x},A)<r\}$. Denote the Borel sigma algebra on $\mathbb{R}^n$ by $\mathbb{B}^n$ and for a set $A \in \mathbb{B}^n$ write $\mathbb{B}^n \cap A=\{B \in \mathbb{B}^n: B \subset A\}$. We write $\lebesgue(\cdot)$ for the Lebesgue measure on $\mathbb{B}^n$.

\section{General regular variation on cones}\label{Sec:background}

In the following, we will make frequent use of an extension of the concept of multivariate regular variation which was introduced in \cite{HuLi06} and \cite{LiReRo14}. For some $m \in \mathbb{N}$, let
$\otimes:(0,\infty)\times[0,\infty)^m \to [0,\infty)^m, (\lambda,\mathbf{x})\mapsto \lambda\otimes\mathbf{x} $ be a ``multiplicatish'' mapping with the following two properties:
\begin{itemize}
 \item[(A1)] the mapping $\otimes$ is continuous,
 \item[(A2)] $1 \otimes \mathbf{x}=\mathbf{x}$ and for $\lambda_1,\lambda_2>0$ we have $\lambda_1\otimes(\lambda_2 \otimes \mathbf{x})=(\lambda_1 \cdot \lambda_2)\otimes \mathbf{x}$ for all $\mathbf{x}\in [0,\infty)^m$.
\end{itemize}

Consider a closed subcone $\mathbb{C}$ of $[0,\infty)^m$ w.r.t.\ this mapping, that is, $\lambda\otimes \mathbb{C} := \{\lambda\otimes \mathbf{x}: \mathbf{x}\in\mathbb{C}\}\subset \mathbb{C} $ for all $\lambda>0$. We assume that the following condition holds:
\begin{itemize}
 \item[(A3)] $d(\mathbf{x},\mathbb{C})< d(\lambda \otimes \mathbf{x},\mathbb{C})$ if $\lambda>1$ and $\mathbf{x} \in \mathbb{O}$.
\end{itemize}
The complement $\mathbb{O}:=[0,\infty)^m \setminus \mathbb{C}$ is an open cone, which is assumed not to be empty.

The notion of regular variation on $\mathbb{O}$ w.r.t.\ $\otimes$, which is introduced below, rests on the definition of convergence in the space $\mathbb{M}_{\mathbb{O}}$ of Borel measures on $(\mathbb{O}, \mathbb{B}^m \cap \mathbb{O})$ whose restrictions to $[0,\infty)^m \setminus \mathbb{C}^r$ are finite for each $r>0$. Denote by $\mathcal{C}^+(\mathbb{O})$  the class of non-negative, bounded and continuous functions $f$ on $\mathbb{O}$ vanishing on $\mathbb{C}^r$ for some $r > 0$.
We endow $\mathbb{M}_{\mathbb{O}}$ with the topology that is generated by open sets of the form
\begin{equation*} \left\{\nu \in  \mathbb{M}_{\mathbb{O}}: \left|\int f_i d\nu - \int f_i d\mu\right|<\epsilon, 1\le i\le k\right\}
\end{equation*}
with $\mu \in \mathbb{M}_{\mathbb{O}}, f_i \in \mathcal{C}^+(\mathbb{O}), i=1, \ldots, k,$ and $\epsilon>0$.  A Portmanteau Theorem (cf.\ \cite{LiReRo14}, Theorem 2.1) shows that convergence of measures $\nu_n$ to a measure $\nu$ in this topology is equivalent to the convergence $\nu_n(A)\to\nu(A)$ for all Borel sets $A$ in $\mathbb{O}$ which are bounded away from $\mathbb{C}$ and for which $\nu(\partial A)=0$.

\begin{definition}[see \cite{LiReRo14}, Definitions 3.1 and 3.2]\label{Def:RVkappa}
A measure $\nu \in \mathbb{M}_{\mathbb{O}}$ is called\emph{ regularly varying on $\mathbb{O}$ with respect to the mapping $\otimes$} if there exists an increasing, regularly varying function $c:[0,\infty)\to (0,\infty)$ and a nonzero measure $\mu \in \mathbb{M}_{\mathbb{O}}$ such that
\begin{equation*}c(x) \nu(x \otimes \cdot) \to \mu(\cdot) \; \mbox{in} \; \mathbb{M}_{\mathbb{O}} \;\;\; \mbox{as}\; x \to \infty.\end{equation*}
\end{definition}

\begin{lemdef}[see \cite{LiReRo14}, Theorem 3.1]\label{LemDef:index} Definition \ref{Def:RVkappa} implies that there exists an $\alpha \geq 0$ such that
\begin{equation}\label{Eq:homlimit}\mu(\lambda \otimes A)=\lambda^{-\alpha} \mu(A)
\end{equation}
for all $\lambda>0$ and Borel sets $A \subset \mathbb{O}$. We call $-\alpha$ the {\em index of regular variation of the measure $\nu$} in Definition \ref{Def:RVkappa}. The value of $\alpha$ in \eqref{Eq:homlimit} is equal to the index of regular variation of the normalizing function $c$ in Definition \ref{Def:RVkappa}.
\end{lemdef}
\begin{proof}
Equation \eqref{Eq:homlimit} is stated in \cite{LiReRo14}, Theorem 3.1. By this and (A2), we have, for all $\lambda>0$ and $A$ in $\mathbb{O}$ which are bounded away from $\mathbb{C}$ and for which $\nu(\partial A)=0$, that
\begin{eqnarray*}
 \lim_{x \to \infty} \frac{c(\lambda x)}{c(x)}=\lim_{x \to \infty} \frac{\nu(x \otimes A)}{\nu((\lambda x) \otimes  A)}=\lim_{x \to \infty} \frac{\nu(x \otimes A)}{\nu(x \otimes (\lambda \otimes A))}= \frac{\mu(A)}{\mu(\lambda\otimes A)}=\lambda^\alpha.
\end{eqnarray*}
Therefore, $c$ is (univariate) regularly varying with index $\alpha$.
\end{proof}

Definition \ref{Def:RVkappa} unifies several different concepts of regular variation of a random vector $\mathbf{X}=(X_1, \ldots, X_m)$ with values in $[0,\infty)^m$ and distribution $\nu$.

\begin{example}[Multivariate regular variation] \label{ex:MRV}
 If $\otimes$ denotes the usual scalar multiplication $\lambda\otimes\mathbf{x}=\lambda\mathbf{x}$ and $\mathbb{C}:=\{\mathbf{0}\}$, then $\mathbb{O}=[0,\infty)^m \setminus\{\mathbf{0}\}$ and  Definition \ref{Def:RVkappa} reads as
$$ c(x)P\left(\left(\frac{X_1}{x}, \ldots, \frac{X_m}{x}\right) \in A \right)\to \mu(A) $$
as $x \to \infty$ for all $\mu$-continuity sets $A\subset\mathbb{B}^m\cap [0,\infty)^m$ bounded away from $\mathbf{0}$. This is the classical regular variation of $\mathbf{X}$ (see e.g.\ \cite{Res07}, Section 6.1.4).
\end{example}

\begin{example}[Ledford-Tawn-model]  \label{ex:LT}
If $\mathbb{C}:=[0,\infty)^m\setminus(0,\infty)^m$, then regular variation on $\mathbb{O}=(0,\infty)^m$ w.r.t.\ the usual scalar multiplication has been considered by \cite{LT97} in the bivariate case $m=2$ (after suitable marginal standardization). It is equivalent to the convergence
$$ \tilde c(x)P\left(\left(\frac{X_1}{x}, \ldots, \frac{X_m}{x}\right) \in A \right)\to \mu(A) $$
as $x \to \infty$ for all $\mu$-continuity sets $A\subset\mathbb{B}^m\cap [0,\infty)^m$ bounded away from both axes in the case $m=2$ resp.\ from $\{\mathbf{x}: x_i=0$ for some $1\le i\le m\}$ in the general case.

Note that a random vector $\mathbf{X}$ may be regularly varying in the classical sense of Example \ref{ex:MRV} and in the present sense with different normalizing functions $c$ resp.\ $\tilde c$. If $c(x)=o(\tilde c(x))$ as $x\to\infty$, then $\mathbf{X}$ is said to exhibit hidden regular variation (cf.\ \cite{Res07}, Section 9.4.1).
\end{example}

Here, we consider a different mapping $\otimes$ and different cones as well. Let, for $\boldsymbol{\kappa} \in [0,\infty)^m$,
\begin{align*} \otimes_{\boldsymbol{\kappa}}: \, & (0,\infty) \times [0,\infty)^m \to [0,\infty)^m, \\
 & (\lambda,(x_1, \ldots,x_m)) \mapsto \lambda \otimes_{\boldsymbol{\kappa}} (x_1, \ldots, x_m):=(\lambda^{\kappa_1}x_1, \ldots, \lambda^{\kappa_m}x_m).
\end{align*}
We want to analyze the asymptotic behavior of $m$-dimensional non-negative random vectors that have extreme values in $n \in\{1,\ldots, m\}$ of their components. For ease of notation, assume that the first $n$ components of $\boldsymbol{\kappa}$ are positive and the last $m-n$ components are equal to zero, so that $x^{\kappa_i}\to \infty$ as $x\to\infty$ only for $1\le i\le n$. Define the cones $\mathbb{C}_n=([0,\infty)^n\setminus (0,\infty)^n)\times [0,\infty)^{m-n}$ and $\mathbb{O}_n=[0,\infty)^m \setminus \mathbb{C}_n=(0,\infty)^n\times[0,\infty)^{m-n}$ w.r.t.\ the mapping $ \otimes_{\boldsymbol{\kappa}}$. Since
$$ d(\mathbf{x},\mathbb{C}_n)=\min\{x_1, \ldots, x_n\} < \min\{\lambda^{\kappa_1}x_1, \ldots, \lambda^{\kappa_n} x_n\}=d(\lambda \otimes_{\boldsymbol{\kappa}} \mathbf{x},\mathbb{C}_n) $$
for all $\lambda>1$ and $\mathbf{x} \in \mathbb{O}_n$, the assumptions (A1)--(A3)  are satisfied. Note that in the case $n=m$ and $\boldsymbol{\kappa}=\boldsymbol{1}$, the regular variation on $\mathbb{O}_n$ w.r.t.\ $\otimes_{\boldsymbol{\kappa}}$ is equivalent to the concept of regular variation considered in Example \ref{ex:LT}.

\begin{lemma}\label{lem:vectorisrv} Let $\boldsymbol{\kappa} \in [0,\infty)^m$ with $\kappa_i>0, 1 \leq i \leq n,$ and $\kappa_i=0, n<i\leq m,$ for some $n \leq m$. Furthermore, let $X_1, \ldots, X_m$ be independent, non-negative random variables such that $X_1, \ldots, X_n$ are regularly varying with index $-1$. Then, $P^{(X_j)_{1 \leq j \leq m }}$ is regularly varying on $\mathbb{O}_n$  w.r.t.\ $\otimes_{\boldsymbol{\kappa}}$ with index $-\alpha=-\sum_{i=1}^m\kappa_i$.

\end{lemma}
\begin{proof}
 From the independence of $X_1, \ldots, X_m$ and the regular variation of $X_1, \ldots, X_n$ it follows that
\begin{align} & \nonumber \lim_{x \to \infty} \frac{P^{(X_j)_{1 \leq j \leq m}}\left(x \otimes_{\boldsymbol{\kappa}} \left(\left(\bigtimes\limits_{1 \leq i \leq n}(a_i,\infty)\right) \times \left(\bigtimes\limits_{n< j \leq m}[b_j,\infty)\right)\right)\right)}{\prod_{i=1}^n P(X_i>x^{\kappa_i})}\\
\nonumber &= \lim_{x \to \infty} \frac{\prod_{i=1}^n P(X_i>a_i x^{\kappa_i})\prod_{j=n+1}^m P(X_j\geq b_j)}{\prod_{i=1}^n P(X_i>x^{\kappa_i})}\\
\label{Eq:definemu} &= \prod_{i=1}^n a_i^{-1} \prod_{j=n+1}^m P(X_j\geq b_j)=:\mu\left(\left(\bigtimes\limits_{1 \leq i \leq n}(a_i,\infty)\right) \times \left(\bigtimes\limits_{n< j \leq m}[b_j,\infty)\right)\right)
\end{align}
for all $a_i>0, 1 \leq i \leq n,$ and $b_j\geq 0, n<j \leq m$. Since these limits are finite, we have shown that the family of measures $c(x) P^{(X_j)_{1 \leq j \leq m }}\left(x \otimes_{\boldsymbol{\kappa}} \cdot\right)$, $ x>0$, is relatively compact in $\mathbb{M}_{\mathbb{O}_n}$ (cf.\ Theorems~2.4 and 2.5 in \cite{LiReRo14}), where $c(x)=\left(\prod_{i=1}^n P(X_i>x^{\kappa_i})\right)^{-1}$. Furthermore, all accumulation points of this family agree on a generating $\pi$-system. Thus, $P^{(X_j)_{1 \leq j \leq m }}$ is regularly varying on $\mathbb{O}_n$ w.r.t.\ $\otimes_{\boldsymbol{\kappa}}$. The index of regular variation follows from Lemma and Definition~\ref{LemDef:index}, since $c$ is regularly varying with index $\sum_{i=1}^n\kappa_i=\sum_{i=1}^m\kappa_i$.
\end{proof}

\section{Joint extremal behavior of random power products}\label{Sec:mainsec}
In the following, let $X_1, \ldots, X_m, m \in \mathbb{N},$ be independent, non-negative random variables, not necessarily with the same distribution. We will give asymptotics for the joint exceedance probabilities \eqref{Eq:powerproductsmot} of $n \leq m$ ``power products''
 \begin{equation}\label{Eq:powerprods} \prod_{j=1}^m X_j^{a_{ij}}, \;\;\; 1 \leq i \leq n,
 \end{equation}
over the same threshold $x$ as $x \to \infty$ for rather general values of $a_{ij} \in \mathbb{R}, 1 \leq i \leq n, 1 \leq j \leq m$. A product may take the value $+\infty$ if $X_j=0$ and $a_{ij}<0$ for some $i,j$, but throughout we use the convention that $+\infty \cdot 0=0$ and $0^0=1$.

In order to derive our results we make some assumptions about the tail behavior of the $X_j, 1 \leq j \leq m$. We assume that all or at least the ``relevant'' (in a sense specified below) $X_j, 1 \leq j \leq m,$ are regularly varying with index $-1$. We will see that the joint extremal behavior of the products in \eqref{Eq:powerprods} is closely related to the solution of the linear optimization problem
\begin{equation}\label{Eq:LOP} \mbox{ find } \mathbf{x} \geq \mathbf{0} \mbox{ such that } \mathbf{A}\mathbf{x} \geq \mathbf{1}, \;\;\; \sum_{i=1}^m x_i \to \min!
\end{equation}
with $\mathbf{A}=(a_{ij})_{1\le i\le n, 1\le j\le m} \in \mathbb{R}^{n \times m}$.
Before we give proofs for the asymptotic behavior of the joint exceedances, we want to motivate the connection of our question to the linear optimization problem in \eqref{Eq:LOP}. To this end, assume for simplicity that $a_{ij}\geq 0$ for all $1 \leq i \leq n, 1 \leq j \leq m$. Let $\mathbf{y} \geq \mathbf{0}$ be a feasible solution to \eqref{Eq:LOP}, i.e.\ $\mathbf{A}\mathbf{y} \geq \mathbf{1}$. Note that for all $x\geq 1$
\begin{equation*} X_j>x^{y_j}, \;\; 1 \leq j \leq m \;\;\; \Rightarrow \;\;\; \prod_{j=1}^m X_j^{a_{ij}}\geq x^{\sum_{j=1}^m a_{ij}y_j}\geq x, \;\;\; 1 \leq i \leq n,\end{equation*}
by \eqref{Eq:LOP} and thus
\begin{equation}\label{Eq:lowerboundheuristic} P\left(\prod_{j=1}^m X_j^{a_{ij}}\geq x, \;1 \leq i \leq n\right)\geq \prod_{j=1}^m P(X_j>x^{y_j}).
\end{equation}
If all $X_j, 1 \leq j \leq m,$ are independent and regularly varying with index $-1$, the right hand side is a regularly varying function in $x$, with index $\alpha=-\sum_{j=1}^m y_j$. Now, the smaller the value of $|\alpha|$, the slower is the decay of the function on the right hand side as $x \to \infty$. So, heuristically, if the value of $|\alpha|$ and thus the value of $\sum_{j=1}^m y_j$ is minimized, this is the most likely combination of extremal events for the single $X_j$ which leads to joint extremal behavior of the power products \eqref{Eq:powerprods}. We will see in Theorem \ref{Th:maintheorem} that the right hand side of \eqref{Eq:lowerboundheuristic} is not only a lower bound for the joint exceedance probabilities but also, under some additional assumptions about real valued $\mathbf{A}$, tail equivalent to it. For a general (not necessarily non-negative) matrix $\mathbf{A}$, the next theorem gives upper and lower bounds for the order of decay of the joint exceedance probabilities.

\begin{theorem}\label{Th:bounds}
Let $X_1, \ldots, X_m$ be independent non-negative random variables. Let $\boldsymbol{\kappa}=(\kappa_1, \ldots, \kappa_m)^T$ be an optimal solution to \eqref{Eq:LOP}.
\begin{itemize}
\item[(a)] Assume that all $X_j, 1 \leq j \leq m,$ are regularly varying with index -1. Then for all $\epsilon>0$,
\begin{equation}\label{Eq:lowerbound}
x^{-\sum_{i=1}^m\kappa_i-\epsilon} = o\left(P\left(\prod_{j=1}^m X_j^{a_{ij}}>x, \;1 \leq i \leq n\right)\right), \;\;\; x \to \infty.
\end{equation}
\item[(b)] Assume that $E(X_j^{1-\delta})<\infty, 1 \leq j \leq m,$ for all $\delta \in (0,1)$, and additionally that there exists $c>0$ such that $P(X_j \geq c)=1$ for all $1 \leq j \leq m $ with $\kappa_j=0$. Then for all $\epsilon>0$,
\begin{equation}\label{Eq:upperbound}
P\left(\prod_{j=1}^m X_j^{a_{ij}}>x, \;1 \leq i \leq n\right)= o(x^{-\sum_{i=1}^m\kappa_i+\epsilon}), \;\;\; x \to \infty.
\end{equation}
\end{itemize}
\end{theorem}
\begin{remark}
In contrast to the other results of this and the following section, the assumptions of Theorem \ref{Th:bounds} (b) do not include regular variation of at least some of the $X_j, 1 \leq j \leq m$. Note, however, that regular variation with index $-1$ of $X_1, \ldots, X_m$ implies that $E(X_j^{1-\delta})<\infty, 1 \leq j \leq m,$ for all $\delta \in (0,1)$.
\end{remark}

The proof is given in Section \ref{Sec:boundsproof}. Under some additional assumptions about the structure of $\mathbf{A}$, the following Theorem \ref{Th:maintheorem} gives precise asymptotics for the joint exceedance probabilities of the random power products.
\begin{theorem}\label{Th:maintheorem}
 Let $\mathbf{A}=(a_{ij}) \in \mathbb{R}^{n \times m}, n \leq m,$ be such that the optimal solution $\boldsymbol{\kappa}$ to the linear optimization problem \eqref{Eq:LOP} is unique and non-degenerate (i.e.\ it has $n$ positive components) and denote by $\mathbf{A}_{\boldsymbol{\kappa}} \in \mathbb{R}^{n \times n}$ the matrix which is derived from $\mathbf{A}$ by deleting all columns $1 \leq j \leq m$ for which $\kappa_j=0$. Then this matrix is invertible.

Let $X_1, \ldots, X_m$ be independent non-negative random variables and assume that there exists $\epsilon>0$ such that for
\begin{align}
\nonumber 1 \leq j \leq m \mbox{ with } \kappa_j>0:& \;\;\; X_j \mbox{ is regularly varying with index } -1 \\
\label{Eq:Cond2Main} 1 \leq j \leq m \mbox{ with } \kappa_j=0:& \;\;\;
\begin{array}{cc}
E\left(X_j^{(\mathbf{1}^T\mathbf{A}_{\boldsymbol{\kappa}}^{-1}\mathbf{A})_j+\epsilon}\right)<\infty  & \mbox{ and }\\
\hspace{0.3cm} E\left(X_j^{(\mathbf{1}^T\mathbf{A}_{\boldsymbol{\kappa}}^{-1}\mathbf{A})_j-\epsilon}\right)<\infty &
\end{array}.
\end{align}
Then
\begin{eqnarray}\label{Eq:Maintheorem1} && \lim_{x \to \infty} \frac{P\left(\prod_{j=1}^m X_j^{a_{ij}}>c_i x, \;1 \leq i \leq n\right)}{\prod\limits_{\{1 \leq j \leq m: \kappa_j>0\}}P(X_j>x^{\kappa_j})}\\
\label{Eq:Maintheorem2} &=&|\det \mathbf{A}_{\boldsymbol{\kappa}}|^{-1}\frac{\prod_{i=1}^n c_i^{-(\mathbf{1}^T\mathbf{A}_{\boldsymbol{\kappa}}^{-1})_i}}{\prod_{i=1}^n (\mathbf{1}^T\mathbf{A}_{\boldsymbol{\kappa}}^{-1})_i}\prod_{j:\kappa_j=0} E\left(X_j^{ (\mathbf{1}^T\mathbf{A}_{\boldsymbol{\kappa}}^{-1}\mathbf{A})_j}\right)=:\mu\left(\bigtimes_{1 \leq i \leq n} (c_i, \infty)\right)
\end{eqnarray}
for all $c_i>0, 1 \leq i \leq n$.
In particular, the distribution of $(\prod_{j=1}^m X_j^{a_{ij}})_{1 \leq i \leq n}$ is regularly varying on $(0,\infty)^n$ w.r.t.\ scalar multiplication (in the sense of Example \ref{ex:LT}). The normalizing function can be chosen as $c(x)= (\prod_{\{1 \leq j \leq m: \kappa_j>0\}}P(X_j>x^{\kappa_j}))^{-1}$ and the corresponding limit measure is $\mu$ as above. The index of regular variation is equal to $-\sum_{j=1}^m \kappa_j=-\mathbf{1}^T\mathbf{A}_{\boldsymbol{\kappa}}^{-1}\mathbf{1}$.
\end{theorem}
\begin{remark}
Given the expression in \eqref{Eq:Maintheorem2}, it is obviously necessary to assume that $E\left(X_j^{(\mathbf{1}^T\mathbf{A}_{\boldsymbol{\kappa}}^{-1}\mathbf{A})_j}\right)<\infty$ for all $1 \leq j \leq m$ with $\kappa_j=0$ in order to ensure that the limit in \eqref{Eq:Maintheorem1} is finite. The actual assumption about the moments of those $X_j$ with $\kappa_j=0$  which is stated in Theorem \ref{Th:maintheorem} is similar to the assumption in Breiman's lemma (cf.\ \cite{Br65}) in that we need a little more than the finiteness of the moments $E\left(X_j^{(\mathbf{1}^T\mathbf{A}_{\boldsymbol{\kappa}}^{-1}\mathbf{A})_j}\right)$ in order to apply a dominated convergence theorem. Furthermore, note that for $(\mathbf{1}^T\mathbf{A}_{\boldsymbol{\kappa}}^{-1}\mathbf{A})_j>0$ only the first moment assumption is necessary, the second follows for $\epsilon>0$ chosen small enough. On the other hand, if $(\mathbf{1}^T\mathbf{A}_{\boldsymbol{\kappa}}^{-1}\mathbf{A})_j<0$, then only the second assumption is
necessary. Only in the case $(\mathbf{1}^T\mathbf{A}_{\boldsymbol{\kappa}}^{-1}\mathbf{A})_j=0$ both assumptions are necessary.
\end{remark}

\begin{remark}\label{Rem:dualcoefs}
\begin{itemize}
\item[(a)] Under the given assumptions the vector $\mathbf{1}^T \mathbf{A}_{\boldsymbol{\kappa}}^{-1}$ which appears in the statement of Theorem \ref{Th:maintheorem} is the transposed of the unique optimal solution to the so-called dual problem to \eqref{Eq:LOP}, which is given by
\begin{equation}\label{Eq:DualLO} \mbox{find } \mathbf{x} \geq \mathbf{0} \; \mbox{such that } \mathbf{A}^T \mathbf{x} \leq \mathbf{1}, \;\;\; \sum_{i=1}^mx_i \to \max!,
\end{equation}
see the proof of Theorem \ref{Th:maintheorem} and \eqref{Eq:identitykappahat} for details. This implies that $(\mathbf{1}^T\mathbf{A}_{\boldsymbol{\kappa}}^{-1}\mathbf{A})^T =\mathbf{A}^T \left(\mathbf{1}^T \mathbf{A}_{\boldsymbol{\kappa}}^{-1}\right)^T \leq \mathbf{1}$ and by the assumed uniqueness and non-degeneracy of the solution even $(\mathbf{1}^T\mathbf{A}_{\boldsymbol{\kappa}}^{-1}\mathbf{A})_j<1$ for those $j$ with $\kappa_j=0$, cf.\ the remark after \eqref{Eq:alldualcases} in the proof of Theorem \ref{Th:maintheorem}. Therefore, the assumptions in \eqref{Eq:Cond2Main} are always satisfied if all $X_j, 1 \leq j \leq m,$ are regularly varying with index $-1$ and bounded away from 0 (or their distributions concentrate sufficiently little mass around 0).

\item[(b)] For a general linear program of the form
\begin{equation}\label{Eq:LOPgeneral} \mbox{ find } \mathbf{x} \geq \mathbf{0} \mbox{ such that } \mathbf{A}\mathbf{x} \geq \mathbf{1}, \;\;\; \sum_{j=1}^m z_j x_j \to \min!
\end{equation}
with optimal solution $\boldsymbol{\kappa}$, the value of $z_j-(\mathbf{1}^T\mathbf{A}_{\boldsymbol{\kappa}}^{-1}\mathbf{A})_j$ is sometimes called the reduced cost of variable $1 \leq j \leq m$. If $\kappa_j=0$ in the optimal solution, then this solution is not affected by a change of $z_j$ in the objective function in \eqref{Eq:LOPgeneral} as long as $z_j>(\mathbf{1}^T\mathbf{A}_{\boldsymbol{\kappa}}^{-1}\mathbf{A})_j$, cf.\ Section 2.5.1 in \cite{Si96}. In the context of Theorem \ref{Th:maintheorem}, the values of $(\mathbf{1}^T\mathbf{A}_{\boldsymbol{\kappa}}^{-1}\mathbf{A})_j$ for $j$ with $\kappa_j=0$ can be interpreted in a similar way, since the left or right tail behavior of these $X_j$ does not influence the extremal behavior of the random products  (except for a possible change in the multiplicative constant of the limit) as long as there exists $\epsilon>0$ such that \begin{eqnarray*}\begin{array}{cc}
    P(X_j>x)= O(x^{-(\mathbf{1}^T\mathbf{A}_{\boldsymbol{\kappa}}^{-1}\mathbf{A})_j-\epsilon}),  & \mbox{ if } (\mathbf{1}^T\mathbf{A}_{\boldsymbol{\kappa}}^{-1}\mathbf{A})_j\geq 0 \\
    P(X_j^{-1}>x)= O(x^{(\mathbf{1}^T\mathbf{A}_{\boldsymbol{\kappa}}^{-1}\mathbf{A})_j-\epsilon}),  & \mbox{ if } (\mathbf{1}^T\mathbf{A}_{\boldsymbol{\kappa}}^{-1}\mathbf{A})_j\leq 0
   \end{array}.
\end{eqnarray*}
\end{itemize}
\end{remark}

The proof of Theorem \ref{Th:maintheorem} is given in detail in Section \ref{Sec:proof}, but we want to lay out briefly the main idea here. To this end, assume w.l.o.g. that the first $n$ components of $\boldsymbol{\kappa}$ are positive and the last $m-n$ components are equal to zero. The main idea is to use the regular variation of the measure $\nu:=\bigotimes_{i=1}^m P^{X_i}$ (where $P^{X_i}$ stands for the law of $X_i$) on $\mathbb{O}_n$ w.r.t.\ $\otimes_{\boldsymbol{\kappa}}$, cf. Lemma~\ref{lem:vectorisrv}. Furthermore, we show that under our assumptions the equality $\mathbf{A} \boldsymbol{\kappa}=\mathbf{1}$ holds. Then,
\begin{eqnarray}\nonumber && P\left(\prod_{j=1}^m X_j^{a_{ij}}>c_i x, \;1 \leq i \leq n\right)\\
\label{Eq:isRVproblem} &=&P\left(\prod_{j=1}^m \left(\frac{X_j}{x^{\kappa_j}}\right)^{a_{ij}}>c_i, \;1 \leq i \leq n\right)=\nu(x \otimes_{\kappa} M) \end{eqnarray}
for
\begin{equation*}M = M(\mathbf{A},\mathbf{c}) := \left\{\mathbf{x} \in [0, \infty)^m: \prod_{j=1}^m x_j^{a_{ij}}>c_i, \; 1 \leq i \leq n\right\}.\end{equation*}
The next step is to apply the Portmanteau Theorem (cf.\ Theorem 2.1 in \cite{LiReRo14}) to show convergence of the right hand side in \eqref{Eq:isRVproblem} under suitable normalization as $x \to \infty$. Note, however, that the set $M$ is not bounded away from $\mathbb{C}_n$ (cf.\ Section \ref{Sec:background}), so we cannot directly apply this argument. As an intermediate step, we therefore have to show that we can replace $M$ by $M \cap (\delta,\infty)^n\times[0,\infty)^{m-n}, \delta>0,$ in \eqref{Eq:isRVproblem} and that under the necessary normalization the difference is negligible as $\delta \searrow 0$.

The following example illustrates the statements of Theorem \ref{Th:maintheorem} and in particular the role of negative coefficients $a_{ij}$. It also demonstrates applications of Theorem \ref{Th:maintheorem} for the extreme value analysis of time series.
\begin{example} Let $(Y_t)_{t \in \mathbb{Z}}$ be a log-linear time series of the form
\begin{equation*}
\ln(Y_t)=\ln(X_{t})-0.5\ln(X_{t+1}), \;\;\; t \in \mathbb{Z},
\end{equation*}
where $X_t, t \in \mathbb{Z},$ are i.i.d.\ and regularly varying with index $-1$. Using Theorem \ref{Th:maintheorem}, we can derive the asymptotics for the probability that three consecutive extreme observations of similar magnitude occur, i.e.\ for $P(Y_1>c_1x, Y_2>c_2x, Y_3>c_3x)$, $c_i>0, i=1,2,3.$ Rewrite this probability as
\begin{equation*}
P(X_1X_{2}^{-0.5}>c_1x, X_2X_{3}^{-0.5}>c_2x, X_3X_{4}^{-0.5}>c_3x)=P\left(\prod_{j=1}^4 X_j^{a_{ij}}>c_i x, \;\; 1 \leq i \leq 3\right)
\end{equation*}
with
\begin{equation*}
\mathbf{A}=(a_{ij})_{1 \leq i \leq 3, 1 \leq j \leq 4}=
\left(
\begin{array}{cccc}
                                                        1 & -0.5 & 0 & 0\\
                                                         0 & 1 & -0.5 & 0 \\
                                                        0 & 0 & 1 & -0.5
\end{array}
\right).
\end{equation*}
The optimal solution to \eqref{Eq:LOP} is then given by $\boldsymbol{\kappa}=(7/4, 3/2, 1, 0)$ and this solution is unique and non-degenerate. Furthermore,
\begin{equation*}(\mathbf{1}^T\mathbf{A}_{\boldsymbol{\kappa}}^{-1}\mathbf{A})_4=\left(\mathbf{1}^T\left(
\begin{array}{ccc}
                                                        1 & -0.5 & 0 \\
                                                         0 & 1 & -0.5 \\
                                                        0 & 0 & 1
\end{array}
\right)^{-1} \mathbf{A}\right)_4=\left((1,3/2,7/4)\mathbf{A}\right)_4=-7/8.
\end{equation*}
Let us first additionally assume that $E(X_4^{-7/8-\epsilon})<\infty$ for some $\epsilon>0$.
Then all assumptions of Theorem \ref{Th:maintheorem} are satisfied and the random vector $(Y_1, Y_2, Y_3)$ is regularly varying on $(0,\infty)^3$ with respect to scalar multiplication. The index of regular variation is equal to $-\sum_{j=1}^4 \kappa_j=-17/4$ and the limit measure $\mu$ is given by
$$\mu(\times_{i=1}^3 (c_i,\infty))= \frac{8}{21}c_1^{-1}c_2^{-3/2}c_3^{-7/4}E(X_4^{-7/8}).$$
Note that the negative exponents do influence the solution of the optimization problem and hence the index of regular variation. If, for instance, in the matrix $\mathbf{A}$ $-0.5$ is replaced with $-0.25$ everywhere, the optimal solution is given by $(21/16,5/4,1,0)$ and the index of regular variation equals $-57/16$.

If the assumption $E(X_4^{-7/8-\epsilon})<\infty$ is not satisfied, Theorem \ref{Th:maintheorem} may still be helpful. For instance, let us assume that $X_4^{-1}$ is regularly varying with index $-1/2$, so that the above moment assumption does not hold. But since this assumption implies that $X_4^{-1/2}$ is regularly varying with index $-1$, we can write the above joint exceedance probability as $P(\prod_{j=1}^4 \tilde{X}_j^{\tilde{a}_{ij}}>c_i x, \;\; 1 \leq i \leq 3)$ for $\tilde{X}_j=X_j, 1 \leq j \leq 3$, $\tilde{X}_4=X_4^{-1/2}$ and
\begin{equation*}
\tilde{\mathbf{A}}=(\tilde{a}_{ij})_{1 \leq i \leq 3, 1 \leq j \leq 4}=
\left(
\begin{array}{cccc}
                                                        1 & -0.5 & 0 & 0\\
                                                         0 & 1 & -0.5 & 0 \\
                                                        0 & 0 & 1 & 1
\end{array}
\right).
\end{equation*}
If we replace $\mathbf{A}$ in \eqref{Eq:LOP} by $\tilde{\mathbf{A}}$, then the optimal solution is given by $\tilde{\boldsymbol{\kappa}}=(3/2,1,0,1)$ and this solution is unique and non-degenerate. Furthermore,
\begin{equation*}(\mathbf{1}^T\tilde{\mathbf{A}}_{\tilde{\boldsymbol{\kappa}}}^{-1}
\tilde{\mathbf{A}})_3=\left(\mathbf{1}^T\left(
\begin{array}{ccc}
                                                        1 & -0.5 & 0 \\
                                                         0 & 1 & 0 \\
                                                        0 & 0 & 1
\end{array}
\right)^{-1} \tilde{\mathbf{A}}\right)_3=\left((1,3/2,1)\tilde{\mathbf{A}}\right)_3=1/4.
\end{equation*}
Since $E(X_3^{1/4+\epsilon})<\infty$ for $\epsilon \in (0,3/4)$ by our above assumptions, we can apply Theorem \ref{Th:maintheorem} also in this case to obtain again that $(Y_1, Y_2, Y_3)$ is regularly varying on $(0,\infty)^3$ with respect to scalar multiplication. But now the index of regular variation is equal to $-\sum_{j=1}^4 \tilde{\kappa}_j=-7/2$ and the limit measure $\tilde{\mu}$ is given by
$$\tilde{\mu}(\times_{i=1}^3 (c_i,\infty))= \frac{2}{3}c_1^{-1}c_2^{-3/2}c_3^{-1}E(X_3^{1/4}).$$
\end{example}

\section{Proofs and auxiliary results}
\subsection{Proof of Theorem \ref{Th:bounds}}
\label{Sec:boundsproof}
\begin{proof}
We start with the proof of (a). The optimal solution $\boldsymbol{\kappa}$ to \eqref{Eq:LOP} lies in the closure of
\begin{equation*} N(\mathbf{A}):=\left\{\mathbf{z} \in \mathbb{R}^m: \mathbf{A}\mathbf{z} > \mathbf{1}\right\}. \end{equation*}
Since the ray $\{\mathbf{z} \in \mathbb{R}^m: \mathbf{z}=(1+\delta)\boldsymbol{\kappa}, \delta >0\}$ is a subset of the open set $N(\mathbf{A})$, for all $\epsilon>0$ there exists  $\epsilon' >0$ such that
\begin{equation*} \bigotimes_{j=1}^m \left(\kappa_j\left(1+\frac{\epsilon}{2\sum_{j=1}^m\kappa_j}\right)-\epsilon',\kappa_j
\left(1+\frac{\epsilon}{2\sum_{j=1}^m\kappa_j}\right)+\epsilon'\right) \subset  N(\mathbf{A}).\end{equation*}
Thus, for $x > 1$,
\begin{eqnarray*}
&& P\left(\prod_{j=1}^m X_j^{a_{ij}}>x, 1 \leq i \leq n\right)\\
&\geq& P\left(\prod_{j=1}^m X_j^{a_{ij}}>x, 1 \leq i \leq n, \mbox{ and }  X_j>0, 1 \leq j \leq m\right)\\
&=& P\left(\left(\frac{\ln(X_j)}{\ln(x)}\right)_{1 \leq j \leq m} \in N(\mathbf{A})\right)\\
&\geq & \prod_{j=1}^m P\left(x^{\kappa_j\left(1+\frac{\epsilon}{2\sum_{j=1}^m\kappa_j}\right)-\epsilon'}<X_j
<x^{\kappa_j\left(1+\frac{\epsilon}{2\sum_{j=1}^m\kappa_j}\right)+\epsilon'}\right).
\end{eqnarray*}
By the regular variation of $X_1, \ldots, X_m$, the expression on the right-hand side is of larger order than
\begin{equation*} \prod_{j=1}^m x^{-\kappa_j\left(1+\frac{\epsilon}{2
\sum_{j=1}^m\kappa_j}\right)+\epsilon'-\frac{\epsilon}{2m}}= x^{-\sum_{j=1}^m\kappa_j-\frac{\epsilon}{2}+m\epsilon'-\frac{\epsilon}{2}} \geq x^{-\sum_{j=1}^m\kappa_j-\epsilon}, \;\;\; x \geq 1,\end{equation*}
which proves (a).

For the proof of (b) let us for simplicity assume that $c\geq 1$ so that $P(X_j\geq 1)=1$ for those $1 \leq j \leq m$ with $\kappa_j=0$. The modifications for general $c>0$ (substitute $X_j/c$ for $X_j$) are simple. Let $\tilde{\mathbf{A}}$ be as in Lemma \ref{Lem:makepositive} (a) (see Section \ref{Sec:auxresults} below), so that we have
\begin{align*}
& & \prod_{j=1}^m X_j^{a_{ij}}>x, & \;\;\;1 \leq i \leq n&\\
&\Rightarrow& \prod_{j=1}^m X_j^{\tilde{a}_{ij}}>x, & \;\;\;1 \leq i \leq n&\\
&\Rightarrow& \prod_{j=1}^m \max(X_j,1)^{\tilde{a}_{ij}}>x, & \;\;\;1 \leq i \leq n,&
\end{align*}
where we have used that $\tilde{a}_{ij}>0$ for all $1 \leq j \leq m$ with $\kappa_j>0$ and $X_j\geq 1$ for $1 \leq j \leq m$ with $\kappa_j=0$. The last inequalities imply that for $x>1$
\begin{align*}
&& \sum_{j=1}^m \tilde{a}_{ij}\frac{\ln(\max(X_j,1))}{\ln(x)}&>1,  \;\;\;1 \leq i \leq n \\
& \Rightarrow & \sum_{j=1}^m \frac{\ln(\max(X_j,1))}{\ln(x)}&\geq \sum_{j=1}^m \kappa_j  ,
\end{align*}
because otherwise $\boldsymbol{\kappa}$ could not be an optimal solution to \eqref{Eq:posLP}, in contrast to Lemma \ref{Lem:makepositive} (a) and our assumptions. Thus, for $x>1$,
\begin{equation}\label{Eq:upperboundproof} P\left(\prod_{j=1}^m X_j^{a_{ij}}>x,\;\;1 \leq i \leq n\right)\leq P\left(\prod_{j=1}^m\max(X_j,1)\geq x^{\sum_{j=1}^m\kappa_j}\right).
\end{equation}
By our assumptions, for all $\delta \in (0,1)$,
\begin{equation*} E\left(\left(\prod_{j=1}^m\max(X_j,1)\right)^{1-\delta}\right)= \prod_{j=1}^m E\left(\max(X_j,1)^{1-\delta}\right) <\infty \end{equation*}
and by the Markov inequality and \eqref{Eq:upperboundproof} we conclude \begin{equation*} P\left(\prod_{j=1}^m X_j^{a_{ij}}>x,\;\;1 \leq i \leq n\right)=O\left(x^{-(\sum_{j=1}^m\kappa_j)(1-\delta)}\right)\end{equation*} for all $\delta \in (0,1)$. Choosing $\delta < \epsilon/\sum_{j=1}^m \kappa_j$ yields \eqref{Eq:upperbound}.
\end{proof}

\subsection{Proof of Theorem \ref{Th:maintheorem}}\label{Sec:proof}
In order to prove Theorem \ref{Th:maintheorem}, we first deal with a setting that covers a slightly more general case for the solution of the linear program \eqref{Eq:LOP} than the one assumed in the statement of Theorem \ref{Th:maintheorem}. The proof of this result is by induction on the number of positive components in the unique optimal solution to \eqref{Eq:LOP}. Several auxiliary results needed for the proof can be found in Section \ref{Sec:auxresults}.

\begin{proposition}\label{Pr:mainprop}
Let $\mathbf{A}=(a_{ij}) \in \mathbb{R}^{n \times m}$ with $n, m \in \mathbb{N}$ be such that the solution $\boldsymbol{\kappa}$ to the linear optimization problem \eqref{Eq:LOP}
is unique with $\mathbf{A}\boldsymbol{\kappa}=\mathbf{1}$. Define $J=\{j \in \{1, \ldots, m\}: \kappa_j>0\}$.

Let $X_1, \ldots, X_m$ be independent non-negative random variables. Assume that
\begin{eqnarray*}
\mbox{ for }  j \in J& : & X_j \mbox{ is regularly varying with index $-1$}, \\
\mbox{ for }  j \in \{1, \ldots, m\} \setminus J& : & P(X_j\geq 1)=1 \mbox{ and } E(X_j^{1-\delta})<\infty \mbox{ for all } \delta \in (0,1).
\end{eqnarray*}
Then,
\begin{eqnarray}\nonumber && \lim_{x \to \infty} \frac{P\left(\prod_{j=1}^m X_j^{a_{ij}}>x, \;1 \leq i \leq n\right)}{\prod\limits_{j \in J}P(X_j>x^{\kappa_j})}\\
\nonumber &=&\int\limits_{M(\mathbf{A})}\prod_{j \in J}x_j^{-2}\lebesgue(\mbox{d}(x_j)_{j \in J}) \otimes P^{(X_j)_{j \notin J}}(\mbox{d}(x_j)_{j \notin J}) \in [0,\infty)
\end{eqnarray}
with $M(\mathbf{A}):=\{(x_1, \ldots, x_m) \in [0,\infty)^m: \prod_{j=1}^m x_j^{a_{ij}}>1, \; 1 \leq i \leq n\}$.
\end{proposition}
\begin{proof}
The proof is by induction on the number $l$ of positive components in the unique optimal solution $\boldsymbol{\kappa}$. Note that $\boldsymbol{\kappa} \geq 0$ and that at least one component of $\boldsymbol{\kappa}$ has to be positive in order to satisfy $\mathbf{A} \boldsymbol{\kappa} \geq \mathbf{1}$.
In the following, we assume w.l.o.g. that the first $l \in \mathbb{N}$ components of $\boldsymbol{\kappa}$ are positive and the last $m-l$ components are equal to zero (if this is not the case, interchange the $X_i$'s and the corresponding columns of $\mathbf{A}$ accordingly).

We start now with the case $l=1$, i.e. $\kappa_1>0$ and $\kappa_j=0$ for $2 \leq j \leq m$. Our assumptions imply that $a_{i1}\kappa_1=1$ for all $1 \leq i \leq n$, i.e. $a_{11}=\ldots=a_{n1}$ and $\kappa_1=a_{i1}^{-1}, 1 \leq i \leq n$. Thus,
\begin{equation}\label{Eq:inductionstart}P\left(\prod_{j=1}^mX_j^{a_{ij}}>x, \;\; 1 \leq i \leq n\right)=P\left(X_1\min_{1 \leq i \leq n}\left(\prod_{j=2}^mX_j^{a_{ij}\kappa_1}\right)>x^{\kappa_1}\right).
\end{equation}
If the linear program
\begin{equation}\label{Eq:LOPwithoutfirstcolumn} \mbox{find } \mathbf{x} \in [0, \infty)^{m-1} \; \mbox{such that } \left(\begin{array}{ccc}
a_{12} & \cdots &  a_{1m} \\
\vdots & \ddots & \vdots \\
a_{n2} & \cdots & a_{nm}
\end{array}\right)\mathbf{x} \geq \mathbf{1}, \;\;\; \sum_{i=1}^{m-1}x_i \to \min!
\end{equation}
has no feasible solution, then
\begin{equation*}
\min_{1 \leq i \leq n}\sum_{j=2}^ma_{ij}x_{j-1}<1, \;\;\; \forall \, \mathbf{x} \in [0, \infty)^{m-1},
\end{equation*}
and thus
\begin{equation}\label{Eq:smallerthane} \min_{1 \leq i \leq n} \left(\prod_{j=2}^mX_j^{a_{ij}\kappa_1}\right)<e^{\kappa_1} \;\;\; \mbox{a.s.},
\end{equation}
because $\ln(X_2), \ldots, \ln(X_m) \geq 0$ almost surely by our assumptions.
On the other hand, if there exists a feasible solution to \eqref{Eq:LOPwithoutfirstcolumn}, then there exists $\epsilon>0$ such that all feasible solutions $\mathbf{x}$ to \eqref{Eq:LOPwithoutfirstcolumn} satisfy $\sum_{i=1}^{m-1} x_i>\kappa_1+\epsilon$, since otherwise there would exist a solution $\mathbf{x'}$ $=(0,x_1', \ldots, x_{m-1}')^T$ $\neq (\kappa_1, 0, \ldots, 0)^T$ to \eqref{Eq:LOP} with $\sum_{j=1}^mx_j'\leq \kappa_1$, in contradiction to our assumptions. Hence, an optimal solution to \eqref{Eq:LOPwithoutfirstcolumn} exists with $\sum_{i=1}^{m-1} x_i>\kappa_1+\epsilon$. By Theorem \ref{Th:bounds} (b) we have
\begin{equation*}
P\left(\min_{1 \leq i \leq n} \left(\prod_{j=2}^mX_j^{a_{ij}\kappa_1}\right)>x\right)=o\left(x^{-1-\epsilon/(2\kappa_1)}\right), \;\;\; x \to \infty.
\end{equation*}
So, whether there exists a solution to \eqref{Eq:LOPwithoutfirstcolumn} or not, we have
\begin{equation*}E\left(\left(\min_{1 \leq i \leq n} \left(\prod_{j=2}^mX_j^{a_{ij}\kappa_1}\right)\right)^{1+\epsilon/(4\kappa_1)}\right)<\infty
\end{equation*}
and we may thus apply Breiman's Lemma, cf.\ \cite{Br65}, to derive the asymptotic behavior of \eqref{Eq:inductionstart} as $x \to \infty$. This gives us
\begin{eqnarray*}
&& \lim_{x \to \infty} \frac{P\left(\prod_{j=1}^mX_j^{a_{ij}}>x, \;\; 1 \leq i \leq n\right)}{P(X_1>x^{\kappa_1})}\\
&=& E\left(\min_{1 \leq i \leq n} \left(\prod_{j=2}^mX_j^{a_{ij}\kappa_1}\right)\right) \\
&=& \int\limits_{\{\mathbf{x}\in [0,\infty)^m:x_1>\max_{1 \leq i \leq n} \left(\prod_{j=2}^mx_j^{-a_{ij}\kappa_1}\right)\}}x_1^{-2}\lebesgue(\mbox{d}x_1)\otimes P^{(X_j)_{2 \leq j \leq m}}(\mbox{d}(x_j)_{2 \leq j \leq m}) \\
&=&  \int\limits_{\{\mathbf{x}\in [0,\infty)^m:\prod_{j=1}^m x_j^{a_{ij}}>1, 1 \leq i \leq n\}}x_1^{-2}\lebesgue(\mbox{d}x_1)\otimes P^{(X_j)_{2 \leq j \leq m}}(\mbox{d}(x_j)_{2 \leq j \leq m}),
\end{eqnarray*}
which concludes the proof in the case $l=1$.

For the induction step, assume that Proposition \ref{Pr:mainprop} holds for all matrices $\mathbf{A}^\ast \in \mathbb{R}^{n^\ast \times m^\ast}, n^\ast, m^\ast \in \mathbb{N},$ for which the corresponding linear program \eqref{Eq:LOP} (with $\mathbf{A}$ replaced by $\mathbf{A}^\ast$) has a unique solution $\boldsymbol{\kappa}^\ast$ and for which $\mathbf{A}^\ast\boldsymbol{\kappa}^\ast=\mathbf{1}$ and at most $l-1 \geq 1$ components of $\boldsymbol{\kappa}^\ast$ are positive. In the following, assume that $\boldsymbol{\kappa}$ has $l$ positive components, again w.l.o.g. the first $l$ ones.

Define the map $\otimes_{\boldsymbol{\kappa}}$ as in Section \ref{Sec:background}. From Lemma \ref{lem:vectorisrv} we get that $P^{(X_j)_{1 \leq j \leq m }}$ is regularly varying on $\mathbb{O}_n$ with respect to $\otimes_{\boldsymbol{\kappa}}$.

Now, with $\mathbf{A}\boldsymbol{\kappa}=\mathbf{1}$ we have
\begin{eqnarray*}
P\left(\prod_{j=1}^m X_j^{a_{ij}}>x, \;1 \leq i \leq n\right)&=&P\left(\prod_{j=1}^m \left(\frac{X_j}{x^{\kappa_j}}\right)^{a_{ij}}>1, \;1 \leq i \leq n\right)\\
&=& P(\mathbf{X} \in x \otimes_{\boldsymbol{\kappa}} M),
\end{eqnarray*}
where $\mathbf{X}=(X_1, \ldots, X_m)$ and
\begin{equation*}M = M(\mathbf{A}) := \left\{\mathbf{x} \in [0, \infty)^m: \prod_{j=1}^m x_j^{a_{ij}}>1, \; 1 \leq i \leq n\right\}.\end{equation*}
For $\delta>0$ write
\begin{eqnarray}\nonumber \frac{P(\mathbf{X} \in x \otimes_{\boldsymbol{\kappa}} M)}{\prod_{i=1}^lP(X_i>x^{\kappa_i})}&=&\frac{P\left(\mathbf{X} \in x \otimes_{\boldsymbol{\kappa}} (M \cap \left((\delta,\infty)^l \times [0,\infty)^{m-l})\right)\right)}{\prod_{i=1}^lP(X_i>x^{\kappa_i})}\\
\label{Eq:boundedandunboundedset} &&  \hspace{-1cm}+\,\frac{P\left(\mathbf{X} \in x \otimes_{\boldsymbol{\kappa}} \left(M \cap \left((\delta,\infty)^l \times [0,\infty)^{m-l}\right)^c\right)\right)}{\prod_{i=1}^lP(X_i>x^{\kappa_i})}.
\end{eqnarray}
We will first show that the second summand in \eqref{Eq:boundedandunboundedset} tends to zero as first $x \to \infty$ and then $\delta \searrow 0$.
Note that
\begin{eqnarray}
\nonumber && \lim_{\delta \searrow 0} \limsup_{x \to \infty} \frac{P\left(\mathbf{X} \in x \otimes_{\boldsymbol{\kappa}} \left(M \cap \left((\delta,\infty)^l \times [0,\infty)^{m-l}\right)^c\right)\right)}{\prod_{i=1}^lP(X_i>x^{\kappa_i})} \\
\label{Eq:summandstozero} &\leq & \sum_{k=1}^l \lim_{\delta \searrow 0} \limsup_{x \to \infty} \frac{P\left(\mathbf{X} \in x \otimes_{\boldsymbol{\kappa}} M, X_k\leq \delta x^{\kappa_k}\right)}{\prod_{i=1}^lP(X_i>x^{\kappa_i})}.
\end{eqnarray}
We will show that all summands in \eqref{Eq:summandstozero} equal zero. To this end, note first that we may apply Lemma \ref{Lem:makepositive} (b) to the matrix $\mathbf{A}$, i.e. there exists a matrix $\tilde{\mathbf{A}}$ such that $\boldsymbol{\kappa}$ as above is the unique solution to the linear program \eqref{Eq:posLP}
with $\tilde{a}_{ij}>0$ for $1 \leq i \leq n$ and $1 \leq j \leq l$ and $\tilde{\mathbf{A}}\boldsymbol{\kappa}=\mathbf{1}$. We have
\begin{align}\nonumber P\left(\mathbf{X} \in x \otimes_{\boldsymbol{\kappa}} M, X_k\leq \delta x^{\kappa_k}\right)&=&
\label{Eq:tildebound} P\left(\prod_{j=1}^m X_j^{a_{ij}}>x, \;1 \leq i \leq n, X_k <\delta x^{\kappa_k}\right) \\
&\leq&  P\left(\prod_{j=1}^m X_j^{\tilde{a}_{ij}}>x, \;1 \leq i \leq n, X_k <\delta x^{\kappa_k}\right)
\end{align}
by Lemma \ref{Lem:makepositive} (b).
For ease of notation, we restrict ourselves to the analysis for the summand $k=1$ in \eqref{Eq:summandstozero}. For $C>0$, use \eqref{Eq:tildebound}, $\tilde{a}_{i1}>0, 1 \leq i \leq n,$ and $\tilde{\mathbf{A}}\boldsymbol{\kappa}=\mathbf{1}$ to write
\begin{eqnarray}
\nonumber && \frac{P\left(\mathbf{X} \in x \otimes_{\boldsymbol{\kappa}} M, X_1\leq \delta x^{\kappa_1}\right)}{\prod_{i=1}^l P(X_i>x^{\kappa_i})} \\
\nonumber &\leq& \frac{P\left(\prod_{j=1}^m X_j^{\tilde{a}_{ij}}>x, \;1 \leq i \leq n, X_1 \leq \delta x^{\kappa_1}\right)}{\prod_{i=1}^l P(X_i>x^{\kappa_i})} \\
\nonumber &= & \frac{\int\limits_{[0,\infty)}P\left(\frac{X_1}{x^{\kappa_1}}>z^{-1}, X_1\leq \delta x^{\kappa_1}\right)P^{\min\limits_{1 \leq i \leq n}\prod\limits_{j=2}^m \left(\frac{X_j}{x^{\kappa_j}}\right)^{\tilde{a}_{ij}/\tilde{a}_{i1}}}(\mbox{d}z)}{\prod_{i=1}^l P(X_i>x^{\kappa_i})} \\
\nonumber &\leq & \frac{\int\limits_{(\delta^{-1}, x^{\kappa_1}/C]}P\left(\frac{X_1}{x^{\kappa_1}}>z^{-1}\right)/P(X_1>x^{\kappa_1})P^{\min\limits_{1 \leq i \leq n}\prod\limits_{j=2}^m \left(\frac{X_j}{x^{\kappa_j}}\right)^{\tilde{a}_{ij}/\tilde{a}_{i1}}}(\mbox{d}z)}{\prod_{i=2}^l P(X_i>x^{\kappa_i})} \\
\nonumber && +\, \frac{\int\limits_{(x^{\kappa_1}/C, \infty)}P\left(\frac{X_1}{x^{\kappa_1}}>z^{-1}, X_1\leq \delta x^{\kappa_1}\right)P^{\min\limits_{1 \leq i \leq n}\prod\limits_{j=2}^m \left(\frac{X_j}{x^{\kappa_j}}\right)^{\tilde{a}_{ij}/\tilde{a}_{i1}}}(\mbox{d}z)}{\prod_{i=1}^l P(X_i>x^{\kappa_i})} \\
\label{Eq:summandsIandII} &=:& I(x,\delta,C) + II(x,\delta,C).
\end{eqnarray}
We deal first with $I=I(x,\delta,C)$. Use $\tilde{a}_{ij}/\tilde{a}_{i1}>0, 1 \leq i \leq n, 2 \leq j \leq l,$ and $\kappa_j=0$ and $P(X_j \geq 1)=1, l< j \leq m,$ to obtain
\begin{equation}\label{Eq:replacelnbylnplus} \min_{1 \leq i \leq n}\sum_{j=2}^m\frac{\tilde{a}_{ij}}{\tilde{a}_{i1}} \ln\left(\frac{X_j}{x^{\kappa_j}}\right) \leq \min_{1 \leq i \leq n}\sum_{j=2}^m\frac{\tilde{a}_{ij}}{\tilde{a}_{i1}} \left(\ln\left(\frac{X_j}{x^{\kappa_j}}\right)\right)_+ \;\; \mbox{a.s.}
\end{equation}
Choose $\epsilon \in (0,1)$ according to Lemma \ref{Lem:boundedbysum} such that the expression on the right hand side of \eqref{Eq:replacelnbylnplus} is a.s.\ bounded by
\begin{equation*}\frac{1-\epsilon}{1+\epsilon} \sum_{j=2}^m  \left(\ln\left(\frac{X_j}{x^{\kappa_j}}\right)\right)_+. \end{equation*}
For this $\epsilon>0$, there exists $C>0$ such that
\begin{equation*} \frac{P\left(\frac{X_1}{x^{\kappa_1}}>z^{-1}\right)}{P(X_1>x^{\kappa_1})}\leq (1+\epsilon)z^{1+\epsilon} \;\;\; \forall\, 1\leq z\leq x^{\kappa_1}/C, \, x > C\end{equation*}
by Potter's bounds applied to $x \mapsto P(X_1>x)$ (cf.\ \cite{BiGoTe87}, Theorem 1.5.6). So, for $\delta\leq 1$, the numerator of $I(x, \delta, C)$ is bounded by
\begin{eqnarray*}
&& \int_{(\delta^{-1}, \infty)}(1+\epsilon)z^{1+\epsilon} P^{\min\limits_{1 \leq i \leq n}\prod\limits_{j=2}^m \left(\frac{X_j}{x^{\kappa_j}}\right)^{\tilde{a}_{ij}/\tilde{a}_{i1}}}(\mbox{d}z)\\
&= & \int_{\tilde{M}(\tilde{\mathbf{A}}, \delta)} (1+\epsilon)\min_{1 \leq i \leq n}\prod\limits_{j=2}^m z_j^{(1+\epsilon)\tilde{a}_{ij}/\tilde{a}_{i1}}P^{\left(\frac{X_j}{x^{\kappa_j}}\right)_{2 \leq j \leq m}}(\mbox{d}\mathbf{z})\\
&\leq & \int_{\tilde{M}(\tilde{\mathbf{A}}, \delta)} (1+\epsilon)\prod_{j=2}^m \max(1,z_j)^{1-\epsilon} P^{\left(\frac{X_j}{x^{\kappa_j}}\right)_{2 \leq j \leq m}}(\mbox{d}\mathbf{z})\\
&\leq & \sum_{(\beta_j)_{2 \leq j \leq m} \in \{0,1-\epsilon\}^{m-1}}\int_{\tilde{M}(\tilde{\mathbf{A}}, \delta)} (1+\epsilon)\prod_{j=2}^m z_j^{\beta_j} P^{\left(\frac{X_j}{x^{\kappa_j}}\right)_{2 \leq j \leq m}}(\mbox{d}\mathbf{z})
\end{eqnarray*}
with
\begin{equation*} \tilde{M}(\tilde{\mathbf{A}}, \delta):=\left\{(z_2, \ldots, z_m) \in [0,\infty)^{m-1}:\min\limits_{1 \leq i \leq n}\prod\limits_{j=2}^m z_j^{\tilde{a}_{ij}/\tilde{a}_{i1}}>\delta^{-1}\right\}.\end{equation*}
Note that $\sum_{k=2}^m \tilde{a}_{ik}\kappa_k>0, 1 \leq i \leq n,$ by our assumptions about $\tilde{\mathbf{A}}$ and $\boldsymbol{\kappa}$ and let
\begin{equation}\label{Eq:Mdeltadef} D(\delta):=\min\limits_{1 \leq i \leq n}\delta^{-\tilde{a}_{i1}/\sum_{k=2}^m\tilde{a}_{ik}\kappa_k}>0
\end{equation}
and
\begin{equation}\label{Eq:atildeprime} \tilde{\mathbf{A}}':=(\tilde{a}_{ij}')_{1 \leq i \leq n, 2 \leq j \leq m}=\left(\frac{\tilde{a}_{ij}}{\sum_{k=2}^m\tilde{a}_{ik}\kappa_k}\right)_{1 \leq i \leq n, 2 \leq j \leq m}.
\end{equation}
Hence,
\begin{eqnarray*}
 \tilde{M}(\tilde{\mathbf{A}}, \delta) &\subset & \left\{(z_2, \ldots, z_m) \in [0,\infty)^{m-1}:\min\limits_{1 \leq i \leq n}\prod\limits_{j=2}^m z_j^{\tilde{a}_{ij}'}>D(\delta)\right\}.\\
\end{eqnarray*}
Note that a feasible solution to the linear program
\begin{equation}\label{Eq:LOPtildeprime} \mbox{find } \mathbf{x}=(x_2, \ldots, x_m)^T \geq \mathbf{0} \; \mbox{such that } \tilde{\mathbf{A}}' \mathbf{x} \geq \mathbf{1}, \;\;\; \sum_{i=2}^mx_i \to \min!
\end{equation}
is given by $\tilde{\boldsymbol{\kappa}}'=(\kappa_2, \ldots, \kappa_m)^T$ with $\tilde{\mathbf{A}}'\boldsymbol{\kappa}'=\mathbf{1}$. Furthermore, this is also the unique optimal solution to \eqref{Eq:LOPtildeprime}, because if there would be another feasible solution $(x_2, \ldots, x_m)^T$ to it with $\sum_{j=2}^m x_j \leq \sum_{j=2}^m\kappa_j$, then $\mathbf{x}':=(\kappa_1, x_2, \ldots, x_m)^T$ would be a solution to \eqref{Eq:posLP} as well because of
\begin{equation*}\tilde{\mathbf{A}}\mathbf{x}'=\left(\left(\tilde{a}_{i1}
\kappa_1+\sum_{j=2}^m\tilde{a}_{ij}x_j\right)_{1 \leq i \leq n}\right)^T\geq \left(\left(\tilde{a}_{i1}\kappa_1+\sum_{k=2}^m\tilde{a}_{ik}\kappa_k\right)_{1 \leq i \leq n}\right)^T = \mathbf{1}, \end{equation*}
as $\tilde{\mathbf{A}}\boldsymbol{\kappa}=\mathbf{1}$.
This would lead to a contradiction to our assumption about the uniqueness of $\boldsymbol{\kappa}$ and Lemma \ref{Lem:makepositive} (b). Thus,
\begin{eqnarray*}
I&\leq& \sum_{(\beta_j)_{2 \leq j \leq m} \in \{0,1-\epsilon\}^{m-1}}\frac{\int\limits_{\left\{\mathbf{z}:\min\limits_{1 \leq i \leq n}\prod\limits_{j=2}^m z_j^{\tilde{a}'_{ij}}>D(\delta)\right\}}(1+\epsilon)\prod_{j=2}^m z_j^{\beta_j} P^{\left(\frac{X_j}{x^{\kappa_j}}\right)_{2 \leq j \leq m}}(\mbox{d}\mathbf{z})}{\prod_{i=2}^l P(X_i>x^{\kappa_i})} \\
&=& \sum_{(\beta_j)_{2 \leq j \leq m} \in \{0,1-\epsilon\}^{m-1}} \frac{\int\limits_{\left\{\mathbf{z}:\min\limits_{1 \leq i \leq n}\prod\limits_{j=2}^m z_j^{\tilde{a}'_{ij}}>1\right\}}\prod_{j=2}^m z_j^{\beta_j} P^{\left(\frac{X_j}{(D(\delta)x)^{\kappa_j}}\right)_{2 \leq j \leq m}}(\mbox{d}\mathbf{z})}{\prod_{i=2}^l P(X_i>(D(\delta)x)^{\kappa_i})} \\
&& \cdot \frac{\prod_{i=2}^l P(X_i>(D(\delta)x)^{\kappa_i})}{\prod_{i=2}^l P(X_i>x^{\kappa_i})}(1+\epsilon)\prod_{j=2}^l D(\delta)^{\beta_j \kappa_j}.
\end{eqnarray*}
Now, by Proposition \ref{Lem:momentsconverge} combined with the induction hypothesis and the properties of $\tilde{\mathbf{A}}'$, the first factor of each summand in the above expression converges to the finite expression
\begin{equation*} \int\limits_{M(\tilde{\mathbf{A}}')}\prod_{j=2}^m x_j^{\beta_j} \prod_{j=2}^l x_j^{-2}\lebesgue(\mbox{d}(x_j)_{2 \leq j \leq l}) \otimes P^{(X_j)_{l< j \leq m}}(\mbox{d}(x_j)_{l< j \leq m}),\end{equation*}
as $x \to \infty,$ whereas the remainder of the expression converges to
\begin{equation*}(1+\epsilon)\prod_{j=2}^l D(\delta)^{-\kappa_j(1-\beta_j)} \end{equation*}
with $\kappa_j(1-\beta_j)>0$ for $2 \leq j \leq l$ by our assumptions. The first limit does not depend on the value of $\delta>0$, while the second converges to 0 as $\delta \searrow 0$ and thus $D(\delta) \to \infty$ by \eqref{Eq:Mdeltadef}. We have thus shown that
\begin{equation*} \lim_{\delta \searrow 0}\limsup_{x \to \infty}I(x,\delta,C)=0\end{equation*}
for $C$ large enough.
Let us now deal with $II=II(x,\delta,C)$ from \eqref{Eq:summandsIandII}. We have
\begin{align}\nonumber II &\leq \frac{P\left(\min\limits_{1 \leq i \leq n}\prod\limits_{j=2}^m \left(\frac{X_j}{x^{\kappa_j}}\right)^{\tilde{a}_{ij}/\tilde{a}_{i1}}>x^{\kappa_1}/C\right)}{\prod_{i=1}^l P(X_i>x^{\kappa_i})}\\
\nonumber &\leq  \frac{P\left(\prod\limits_{j=2}^m \max(1,X_j)^{\tilde{a}_{ij}}>x/C^{\tilde{a}_{i1}}, \;\; 1 \leq i \leq n\right)}{\prod_{i=1}^l P(\max(1,X_i)>x^{\kappa_i})}\cdot \frac{\prod_{i=1}^l P(\max(1,X_i)>x^{\kappa_i})}{\prod_{i=1}^l P(X_i>x^{\kappa_i})} \\
 \label{Eq:boundforII} &\leq  \frac{P\left(\prod\limits_{j=2}^m \max(1,X_j)^{\tilde{a}_{ij}}>x D'(C), \;\; 1 \leq i \leq n\right)}{\prod_{i=1}^l P(\max(1,X_i)>x^{\kappa_i})}\cdot \frac{\prod_{i=1}^l P(\max(1,X_i)>x^{\kappa_i})}{\prod_{i=1}^l P(X_i>x^{\kappa_i})},
\end{align}
where we set
\begin{equation*} D'(C):=\min_{1 \leq i \leq n} C^{-\tilde{a}_{i1}}>0 \end{equation*}
for abbreviation.
Set
\begin{equation*} \tilde{\mathbf{A}}'':=(\tilde{a}_{ij})_{1 \leq i \leq n, 2 \leq j \leq m} \in \mathbb{R}^{n \times (m-1)} \end{equation*}
and consider the linear program
\begin{equation}\label{Eq:LOPtildeprimeprime} \mbox{find } \mathbf{x}=(x_2, \ldots, x_m)^T \geq \mathbf{0} \; \mbox{such that } \tilde{\mathbf{A}}'' \mathbf{x} \geq \mathbf{1}, \;\;\; \sum_{i=2}^mx_i \to \min!
\end{equation}
If this linear program has no feasible solution then
\begin{equation*}\min_{1 \leq i \leq n}\prod_{j=2}^m\max(1,X_j)^{\tilde{a}_{ij}}<e \;\; \mbox{a.s.} \end{equation*}
(cf.\ \eqref{Eq:smallerthane} for analogous reasoning), so the first factor in \eqref{Eq:boundforII} equals 0 for $x$ large enough. On the other hand, if there exists a feasible solution to \eqref{Eq:LOPtildeprimeprime}, then there exists $\epsilon>0$ such that all feasible solutions $(x_2, \ldots, x_m)^T$ to \eqref{Eq:LOPtildeprimeprime} satisfy $\sum_{j=2}^m x_j>\sum_{j=1}^m\kappa_j + \epsilon$, since otherwise there would exist a solution $\mathbf{x'}$ $=(0,x_2, \ldots, x_m)^T$ $\neq \boldsymbol{\kappa}$ to \eqref{Eq:posLP} with $\sum_{j=1}^mx_j'\leq \sum_{j=1}^m\kappa_j$, in contradiction to our assumptions and Lemma \ref{Lem:makepositive} (b). In the latter case, the numerator of the first factor in \eqref{Eq:boundforII} is of smaller order than $x^{-\sum_{j=1}^m\kappa_j-\epsilon/2}$ as $x \to \infty$ by Theorem \ref{Th:bounds} (b), while the denominator is regularly varying in $x$ with index $-\sum_{j=1}^m\kappa_j$ and the second factor in \eqref{Eq:boundforII} equals 1 for $x \geq 1$. So, in both
cases,
and for all $C>0$
\begin{equation*} \lim_{\delta \searrow 0}\limsup_{x \to \infty}II(x,\delta,C)=\limsup_{x \to \infty}II(x,\delta,C)=0,\end{equation*}
and the first summand in \eqref{Eq:summandstozero} is equal to zero. All other summands can be treated analogously.

Taken together, we have shown that
\begin{equation}\label{Eq:unboundedisfinite}
\lim_{\delta \searrow 0} \limsup_{x \to \infty} \frac{P\left(\mathbf{X} \in x \otimes_{\boldsymbol{\kappa}} \left(M \cap \left((\delta,\infty)^l \times [0,\infty)^{m-l}\right)^c\right)\right)}{\prod_{i=1}^l P(X_i>x^{\kappa_i})}=0.
\end{equation}
With $\mu(\cdot)$ as defined in \eqref{Eq:definemu} we have
\begin{eqnarray*}
 && \lim_{x \to \infty} \frac{P\left(\prod_{j=1}^mX_j^{a_{ij}}>x, 1 \leq i \leq n\right)}{\prod_{i=1}^lP(X_i>x^{\kappa_i})}= \lim_{x \to \infty} \frac{P(\mathbf{X} \in x \otimes_{\boldsymbol{\kappa}} M)}{\prod_{i=1}^lP(X_i>x^{\kappa_i})}\\
 &=& \lim_{\delta \searrow 0} \lim_{x \to \infty}  \frac{P\left(\mathbf{X} \in x \otimes_{\boldsymbol{\kappa}} (M \cap \left((\delta,\infty)^l \times [0,\infty)^{m-l})\right)\right)}{\prod_{i=1}^lP(X_i>x^{\kappa_i})}\\
 &=& \lim_{\delta \searrow 0} \mu\left(M \cap \left((\delta,\infty)^l \times [0,\infty)^{m-l})\right)\right)\\
 &=& \lim_{\delta \searrow 0} \int\limits_{M \cap \left((\delta,\infty)^l \times [0,\infty)^{m-l}\right)}\prod\limits_{j=1}^l x_l^{-2}\lebesgue(\mbox{d}(x_j)_{1 \leq j \leq l})\otimes P^{(X_j)_{l < j \leq m}}(\mbox{d}(x_j)_{l< j \leq m}) \\
 &=& \int\limits_{M}\prod\limits_{j=1}^l x_l^{-2}\lebesgue(\mbox{d}(x_j)_{1 \leq j \leq l})\otimes P^{(X_j)_{l< j \leq m}}(\mbox{d}(x_j)_{l< j \leq m}),
\end{eqnarray*}
by monotone convergence where we used that
\begin{equation*} M \cap\left((0,\infty)^l \times [0,\infty)^{m-l}\right)=M, \end{equation*}
because $\kappa_j>0$ for $1 \leq j \leq l$ implies that at least one $1 \leq i \leq n$ exists with $a_{ij}>0$. But then $\min_{1 \leq i \leq n} \prod_{l=1}^m X_l^{a_{il}}=0$ as soon as $X_j=0, 1 \leq j \leq l$, so $M \subset \left((0,\infty)^l \times [0,\infty)^{m-l}\right)$.

This concludes the proof of Proposition \ref{Pr:mainprop}.
\end{proof}
\begin{proof}[Proof of Theorem \ref{Th:maintheorem}]
Let us again assume w.l.o.g.\ that the first $n \geq 1$ components of $\boldsymbol{\kappa}$ are positive and the last $m-n\geq 0$ components are equal to zero. We start with some implications of our assumptions about the matrix $\mathbf{A}$. Since the optimal solution $\boldsymbol{\kappa}$ to \eqref{Eq:LOP} is unique it must be a vertex of the polygon defined by $\{\mathbf{x} \geq \mathbf{0}:\mathbf{A}\mathbf{x}\geq \mathbf{1}\}$, cf.\ \cite{Si96}, Theorem 1.5. Each vertex of $\{\mathbf{x} \geq \mathbf{0}:\mathbf{A}\mathbf{x}\geq \mathbf{1}\}$ corresponds to a so-called basic feasible solution (cf.\ \cite{Si96}, Theorem 1.2) of the standard form linear program
\begin{equation}\label{Eq:LOPstandard} \mbox{find } \mathbf{x} \in \mathbb{R}^{m+n}\; \mbox{such that }  (\mathbf{A};(-1)\cdot \mathbf{E}_n)\mathbf{x} =\mathbf{1}, \;\; \mathbf{x} \geq \mathbf{0},\;\;\; \sum_{i=1}^m x_i \to \min!,
\end{equation}
where the matrix $(\mathbf{A},(-1)\cdot \mathbf{E}_n) \in \mathbb{R}^{n \times (m+n)}$ consists of the columns of $\mathbf{A}$ in its first $m$ columns and of the columns of the $n$-dimensional unit matrix, $\mathbf{E}_n$, multiplied with $-1$ in its last $n$ columns. The basic feasible solutions of \eqref{Eq:LOPstandard} can be found by choosing $n$ linearly independent columns of $(\mathbf{A};(-1)\cdot \mathbf{E}_n)$ with indices $B \subset \{1, \ldots, m+n\}$, denoting the resulting matrix by $(\mathbf{A};(-1)\cdot \mathbf{E}_n)_B$ and deriving $\mathbf{s}_B=((s_j)_{j \in B})^T:=((\mathbf{A};(-1)\cdot \mathbf{E}_n)_B)^{-1}\mathbf{1}$. If $\mathbf{s}_B \geq \mathbf{0}$, then we call
\begin{equation*}\mathbf{x}_B=(x_1, \ldots, x_{m+n})^T \;\;\; \mbox{with} \;\; \begin{cases}
                                                              x_j=s_j, \;\; &\mbox{if } j \in B, \\
                                                              x_j=0, \;\; &\mbox{if } j \in \{1, \ldots, m+n\} \setminus B
                                                              \end{cases} \end{equation*}
a basic feasible solution to \eqref{Eq:LOPstandard}. The corresponding solution to \eqref{Eq:LOP} is given by the first $m$ components of $\mathbf{x}_B$, the remaining last $n$ components of $\mathbf{x}_B$ are called slack variables. Since we assumed that the first $n$ components of $\boldsymbol{\kappa}$ are positive and that the optimal solution is unique, it can only correspond to the basic feasible solution with $B=\{1, \ldots, n\}$ which implies that $\mathbf{A}_{\boldsymbol{\kappa}}$, the matrix which consists of only the first $n$ columns of $\mathbf{A}$, is invertible and $(\kappa_1, \ldots, \kappa_n)^T=(\mathbf{A}_{\boldsymbol{\kappa}})^{-1}\mathbf{1}$ which leads to $\mathbf{A}\boldsymbol{\kappa}=\mathbf{1}$. Thus, the assumptions about $\mathbf{A}$ of Theorem \ref{Th:maintheorem} are a special case of the assumptions about $\mathbf{A}$ of Proposition \ref{Pr:mainprop}. Furthermore, the optimal value of \eqref{Eq:LOP} equals
\begin{equation}\label{Eq:optval}
\sum_{j=1}^m\kappa_j=\sum_{j=1}^n\kappa_j=\mathbf{1}^T(\mathbf{A}_{\boldsymbol{\kappa}}^{-1})\mathbf{1}. 
\end{equation}

Since we assumed $\boldsymbol{\kappa}$ to be unique and non-degenerate, the optimal solution to the dual problem \eqref{Eq:DualLO} is unique and non-degenerate as well, cf.\ \cite{Si96}, Theorem 2.11. Furthermore, the optimal solution $\hat{\boldsymbol{\kappa}}$ to \eqref{Eq:DualLO} is in our case given by
\begin{equation}\label{Eq:identitykappahat} \hat{\boldsymbol{\kappa}}=(\mathbf{A}_{\boldsymbol{\kappa}}^{-1})^T\mathbf{1},
\end{equation}
cf.\ \cite{Si96}, Theorem 2.2. This explains Remark~\ref{Rem:dualcoefs}~(a).

Again w.l.o.g.\ assume in the following that $(\mathbf{1}^T\mathbf{A}_{\boldsymbol{\kappa}}^{-1}\mathbf{A})_j=0$ for $n<j \leq n'$ with $n \leq n' \leq m$ and that $(\mathbf{1}^T\mathbf{A}_{\boldsymbol{\kappa}}^{-1}\mathbf{A})_j\neq 0$ for $n'<j \leq m$.
Define now for $1 \leq j \leq m$
\begin{equation}\label{Eq:DefXtilde} \hat{X}_j=\begin{cases}
               X_j, & \mbox{ for } j \leq n, \\
               X_j^{\epsilon}, & \mbox{ for } n< j \leq n', \\
               X_j^{(\mathbf{1}^T\mathbf{A}_{\boldsymbol{\kappa}}^{-1}\mathbf{A})_j+\epsilon}, & \mbox{ for } j>n', (\mathbf{1}^T\mathbf{A}_{\boldsymbol{\kappa}}^{-1}\mathbf{A})_j > 0, \\
               X_j^{(\mathbf{1}^T\mathbf{A}_{\boldsymbol{\kappa}}^{-1}\mathbf{A})_j-\epsilon}, & \mbox{ for } j>n', (\mathbf{1}^T\mathbf{A}_{\boldsymbol{\kappa}}^{-1}\mathbf{A})_j < 0,
               \end{cases} \end{equation}
with $\epsilon>0$ as in the statement of Theorem \ref{Th:maintheorem} and set furthermore $\hat{\mathbf{A}}$ with
\begin{equation}\label{Eq:DefAtilde}\hat{a}_{ij}=\begin{cases}
                 a_{ij}, & \mbox{ for } 1 \leq i \leq n, j \leq n, \\
                 \frac{a_{ij}}{\epsilon}, & \mbox{ for } 1 \leq i \leq n, n<j \leq n', \\
                 \frac{a_{ij}}{(\mathbf{1}^T\mathbf{A}_{\boldsymbol{\kappa}}^{-1}\mathbf{A})_j+\epsilon}, &  \mbox{ for } 1 \leq i \leq n, j>n', (\mathbf{1}^T\mathbf{A}_{\boldsymbol{\kappa}}^{-1}\mathbf{A})_j > 0, \\
                 \frac{a_{ij}}{(\mathbf{1}^T\mathbf{A}_{\boldsymbol{\kappa}}^{-1}\mathbf{A})_j-\epsilon}, &  \mbox{ for } 1 \leq i \leq n, j>n', (\mathbf{1}^T\mathbf{A}_{\boldsymbol{\kappa}}^{-1}\mathbf{A})_j < 0.
                 \end{cases}
\end{equation}
 Obviously, this leads to
 \begin{equation}\label{Eq:firsttransfo}P\left(\prod_{j=1}^mX_j^{a_{ij}}>c_ix, \; 1 \leq i \leq n\right)=P\left(\prod_{j=1}^m\hat{X}_j^{\hat{a}_{ij}}>c_ix, 1 \leq i \leq n\right). \end{equation}
In order to apply Proposition \ref{Pr:mainprop} we will distinguish between events where $\hat{X}_j\geq 1$ and those where $\hat{X}_j<1$ for $n< j \leq n'$. Therefore, for $J \subset \{n+1, \ldots, n'\}$ denote the event $\{\hat{X}_j \geq 1 \mbox{ for } j \in \{n+1, \ldots, n'\} \setminus J, \hat{X}_j < 1 \mbox{ for } j \in J\}$ by $B(J)$, and write
\begin{eqnarray}
\nonumber && P\left(\prod_{j=1}^m \hat{X}_j^{\hat{a}_{ij}}>c_i x, \; 1 \leq i \leq n\right)\\
\label{Eq:condsummands}&=& \sum_{J \subset \{n+1, \ldots, n'\}} P\left(\prod_{j=1}^m \hat{X}_j^{\hat{a}_{ij}}>c_ix, \; 1 \leq i \leq n \, \Bigg| \, B(J)\right)P(B(J)).
\end{eqnarray}
For $J$ with $P(B(J))>0$ define now independent random variables $\hat{X}_j^{(J)}, 1 \leq j \leq m$, with
\begin{equation}\label{Eq:DefXJtilde} P^{\hat{X}_j^{(J)}}=\begin{cases}
                        P^{\hat{X}_j}, & \mbox{ for } j \in \{1, \ldots, n, n'+1, \ldots, m\}, \\
                        P^{\hat{X}_j^{-1}|\hat{X}_j<1}, & \mbox{ for } j \in J, \\
                        P^{\hat{X}_j|\hat{X}_j\geq 1}, & \mbox{ for } j \in \{n+1, \ldots, n'\} \setminus J. \\
                        \end{cases}
\end{equation}
Furthermore, set $\hat{\mathbf{A}}^{(J)}$ with
\begin{equation}\label{Eq:DefAJtilde} \hat{a}_{ij}^{(J)}=\begin{cases}
                    \hat{a}_{ij}, & \mbox{ for } 1 \leq i \leq n, j \in \{1, \ldots, m\} \setminus J, \\
                    -\hat{a}_{ij}, & \mbox{ for } 1 \leq i \leq n, j \in J.
                    \end{cases}
\end{equation}
By independence of the $\hat{X}_j$'s and of the $\hat{X}_j^{(J)}$'s, the first factor of each summand in \eqref{Eq:condsummands} is equal to
\begin{equation}\label{Eq:probwithX^{(J)}} P\left(\prod_{j=1}^m \left(\hat{X}_j^{(J)}\right)^{\hat{a}^{(J)}_{ij}}>c_i x, \; 1 \leq i \leq n\right)
\end{equation}
for all $J\subset \{n+1, \ldots, n'\}$ with $P(B(J))>0$. The vector $\boldsymbol{\kappa}$ is a basic feasible solution to
\begin{equation}\label{Eq:LO(J)} \mbox{find } \mathbf{x} \geq \mathbf{0} \; \mbox{such that } \hat{\mathbf{A}}^{(J)} \mathbf{x} \geq \mathbf{1}, \;\;\; \sum_{i=1}^mx_i \to \min!,
\end{equation}
because $\hat{\mathbf{A}}^{(J)}\boldsymbol{\kappa}=\mathbf{1}$, as the first $n$ columns of $\hat{\mathbf{A}}^{(J)}$ are identical to those of $\mathbf{A}$. Furthermore,
\begin{equation}\label{Eq:alldualcases} (\mathbf{1}^T\mathbf{A}_{\boldsymbol{\kappa}}^{-1}\hat{\mathbf{A}}^{(J)})_j=
\begin{cases}
(\mathbf{1}^T\mathbf{A}_{\boldsymbol{\kappa}}^{-1}\mathbf{A})_j=1, & \mbox{ if } j \leq n, \\
\epsilon^{-1}(\mathbf{1}^T\mathbf{A}_{\boldsymbol{\kappa}}^{-1}\mathbf{A})_j=0, & \mbox{ if } j \in \{n+1, \ldots, n'\}\setminus J\\
\epsilon^{-1}(\mathbf{1}^T\mathbf{A}_{\boldsymbol{\kappa}}^{-1}(-\mathbf{A}))_j=0, & \mbox{ if } j \in J\\
\frac{(\mathbf{1}^T\mathbf{A}_{\boldsymbol{\kappa}}^{-1}
\mathbf{A})_j}{(\mathbf{1}^T\mathbf{A}_{\boldsymbol{\kappa}}^{-1}\mathbf{A})_j+\epsilon}\in (0,1), & \mbox{ if } j>n', (\mathbf{1}^T\mathbf{A}_{\boldsymbol{\kappa}}^{-1}\mathbf{A})_j> 0, \\
\frac{(\mathbf{1}\mathbf{A}_{\boldsymbol{\kappa}}^{-1}
\mathbf{A})_j}{(\mathbf{1}\mathbf{A}_{\boldsymbol{\kappa}}^{-1}\mathbf{A})_j-\epsilon} \in (0,1), & \mbox{ if } j>n', (\mathbf{1}^T\mathbf{A}_{\boldsymbol{\kappa}}^{-1}\mathbf{A})_j< 0,
                                                                          \end{cases}
\end{equation}                                                                          
which proves that $\boldsymbol{\kappa}$ is the unique optimal solution to \eqref{Eq:LO(J)}, because $1-(\mathbf{1}^T\mathbf{A}_{\boldsymbol{\kappa}}^{-1}\hat{\mathbf{A}}^{(J)})_j$ is strictly positive for all non-basic variables $n<j\leq m$ (cf.\ the analogue of Theorem 1.6 and the remark after the proof of this theorem in \cite{Si96} for a linear minimization problem instead of a maximization problem).

Let now $J \subset \{n+1, \ldots, n'\}$ and let $\tilde{\hat{\mathbf{A}}}^{(J)}$ be the matrix described in Lemma \ref{Lem:makepositive} (c), corresponding to $\hat{\mathbf{A}}^{(J)}$ with $\tilde{\hat{a}}_{ij}^{(J)}>0$ for all $1 \leq i \leq n, j \in \{1, \ldots, n, n'+1, \ldots, m\}$. Then, for $c>0$,
\begin{align}
\nonumber & P\left(\prod_{j=1}^m \left(\hat{X}_j^{(J)}\right)^{\hat{a}^{(J)}_{ij}}>c_i x, \; 1 \leq i \leq n\right)\\
\nonumber = & P\left(\prod_{j=1}^m \left(\hat{X}_j^{(J)}\right)^{\hat{a}^{(J)}_{ij}}>c_i x, \; 1 \leq i \leq n, \hat{X}_j^{(J)} \geq c, n'< j \leq m\right) \\
\label{Eq:twoXhatsummands} + & \, P\left(\prod_{j=1}^m \left(\hat{X}_j^{(J)}\right)^{\hat{a}^{(J)}_{ij}}>c_i x, \; 1 \leq i \leq n,  \exists \; n'< j \leq m:\hat{X}_j^{(J)} < c \right)
\end{align}
and for the second summand we have
\begin{eqnarray*}
 && P\left(\prod\limits_{j=1}^m \left(\hat{X}_j^{(J)}\right)^{\hat{a}^{(J)}_{ij}}>c_i x, \; 1 \leq i \leq n, \exists \; n'< j \leq m:\hat{X}_j^{(J)} < c\right)\\
 &\leq & \sum_{\emptyset \neq K \subset \{n'+1, \ldots, m\}} P\Bigg(\prod\limits_{j=1}^m \left(\hat{X}_j^{(J)}\right)^{\tilde{\hat{a}}^{(J)}_{ij}}>x \min\limits_{1 \leq k \leq n}c_k, \\
 && \hspace{3cm} \hat{X}_j^{(J)}<c, j \in K, \hat{X}_j^{(J)}\geq c, j \in \{n'+1, \ldots, m\} \setminus K\Bigg)\\
 &\leq & \sum_{\emptyset \neq K \subset \{n'+1, \ldots, m\}}P\Bigg(\prod\limits_{j \in \{1, \ldots, n'\}} \left(\hat{X}_j^{(J)}\right)^{\tilde{\hat{a}}^{(J)}_{ij}}\prod\limits_{j \in \{n'+1, \ldots, m\} \setminus K} \max\left(1,\hat{X}_j^{(J)}\right)^{\tilde{\hat{a}}^{(J)}_{ij}}\\
 && > x  \min\limits_{1 \leq k \leq n}c_k\left( \min\limits_{1 \leq i \leq n} c^{-\sum\limits_{j \in K}\tilde{\hat{a}}_{ij}^{(J)}}\right), \; 1 \leq i \leq n\; \Bigg).
\end{eqnarray*}
Set
\begin{equation*} D(c)=D(c,c_1, \ldots, c_n):=\left(\min\limits_{1 \leq k \leq n}c_k\right)\left( \min\limits_{1 \leq i \leq n} c^{-\sum\limits_{j \in K}\tilde{\hat{a}}_{ij}^{(J)}}\right)>0\end{equation*}
for abbreviation and note that by our assumptions and Lemma \ref{Lem:makepositive} (c) $\boldsymbol{\kappa}$ is the unique optimal solution to
\begin{equation*} \mbox{find } \mathbf{x} \geq \mathbf{0} \; \mbox{such that } \tilde{\hat{\mathbf{A}}}^{(J)} \mathbf{x} \geq \mathbf{1}, \;\;\; \sum_{i=1}^mx_i \to \min!,\end{equation*}
with $\tilde{\hat{\mathbf{A}}}^{(J)}\boldsymbol{\kappa}=\mathbf{1}$, that $P^{\hat{X}_j^{(J)}}=P^{X_j}, 1 \leq j \leq n$, that $\hat{X}_{n+1}^{(J)}, \ldots, \hat{X}_{n'}^{(J)}\geq 1$ a.s. with $E(\hat{X}_j^{(J)})<\infty, n<j \leq n',$ and that $E(\max(1,\hat{X}_j^{(J)}))<\infty, n'<j \leq m$. Therefore, apply Proposition \ref{Pr:mainprop} to obtain for $\emptyset \neq K \subset \{n'+1, \ldots, m\}$
\begin{eqnarray*}
&& \frac{P\Big(\prod\limits_{j \in \{1, \ldots, n'\}} \left(\hat{X}_j^{(J)}\right)^{\tilde{\hat{a}}^{(J)}_{ij}}\prod\limits_{j \in \{n'+1, \ldots, m\} \setminus K} \max\left(1,\hat{X}_j^{(J)}\right)^{\tilde{\hat{a}}^{(J)}_{ij}}>D(c) x, \; 1 \leq i \leq n\Big)}{\prod_{j=1}^nP(X_j>x^{\kappa_j})} \\
& \to & (D(c))^{-\sum_{j=1}^n\kappa_j} D(J),
\end{eqnarray*}
as $x \to \infty$ for some finite constant $D(J)$ which does not depend on $c$. As $D(c) \to \infty$ for $c \searrow 0$ we conclude from \eqref{Eq:twoXhatsummands} that
\begin{align}
\nonumber & \lim_{x \to \infty} \frac{P\left(\prod\limits_{j=1}^m \left(\hat{X}_j^{(J)}\right)^{\hat{a}^{(J)}_{ij}}>c_i x, \; 1 \leq i \leq n\right)}{\prod_{j=1}^nP(X_j>x^{\kappa_j})}\\
\label{Eq:limitctozero}&= \lim_{c \searrow 0}\lim_{x \to \infty} \frac{P\left(\prod\limits_{j=1}^m \left(\hat{X}_j^{(J)}\right)^{\hat{a}^{(J)}_{ij}}>c_i x, \; 1 \leq i \leq n, \hat{X}_j^{(J)} \geq c, n'< j \leq m\right)}{\prod_{j=1}^nP(X_j>x^{\kappa_j})}
\end{align}
Let now
\begin{equation*} c_i(c)=c_i c^{-\sum_{j={n'+1}}^m \hat{a}_{ij}^{(J)}}=c_i c^{-\sum_{j={n'+1}}^m \hat{a}_{ij}}, \;\;\; 1 \leq i \leq n. \end{equation*}
Then
\begin{eqnarray}
\nonumber && \lim_{x \to \infty} \frac{P\left(\prod\limits_{j=1}^m \left(\hat{X}_j^{(J)}\right)^{\hat{a}^{(J)}_{ij}}>c_i x, \; 1 \leq i \leq n, \hat{X}_j^{(J)} \geq c, n'< j \leq m\right)}{\prod_{j=1}^nP(X_j>x^{\kappa_j})} \\
\label{Eq:limitwithcs} &=& \lim_{x \to \infty} \prod_{j=n'+1}^m P(\hat{X}_j^{(J)} \geq c)\left(\prod_{j=1}^n P(X_j>x^{\kappa_j})\right)^{-1}\\
\nonumber && P\left(\prod\limits_{j=1}^{n'} \left(\hat{X}_j^{(J)}\right)^{\hat{a}^{(J)}_{ij}}\prod\limits_{j=n'+1}^{m} \left(\frac{\hat{X}_j^{(J)}}{c}\right)^{\hat{a}^{(J)}_{ij}}>c_i(c) x, \; 1 \leq i \leq n \Bigg| \min\limits_{n'<j\leq m} \hat{X}_j^{(J)} \geq c \right).
\end{eqnarray}
The last factor in the above expression can be written as
\begin{equation*} P\left(\prod\limits_{j=1}^{m} \left(\hat{X}_j^{(J,c)}\right)^{\hat{a}^{(J)}_{ij}}>c_i(c) x, \; 1 \leq i \leq n\right) \end{equation*}
where $\hat{X}_j^{(J,c)}, 1 \leq j \leq m,$ denote independent random variables with
\begin{equation*} P^{\hat{X}_j^{(J,c)}}=\begin{cases}
                                   P^{\hat{X}_j^{(J)}}, & \mbox{ for } 1 \leq j \leq n', \\
                                   P^{c^{-1}\hat{X}_j^{(J)}|\tilde{X}_j^{(J)}\geq c}, & \mbox{ for }  n'< j \leq m.
                                   \end{cases} \end{equation*}
Set
\begin{equation*} \hat{c}_j(c)=\begin{cases}
\prod_{i=1}^n c_i(c)^{(\mathbf{A}_{\boldsymbol{\kappa}}^{-1})_{ji}}, & 1 \leq j \leq n, \\
1, & n<j \leq m,
\end{cases}\end{equation*}
which implies that
\begin{equation*} \prod_{j=1}^m \hat{c}_j(c)^{\hat{a}_{ij}^{(J)}}=\prod_{j=1}^m \hat{c}_j(c)^{a_{ij}}=\exp\left(\sum_{j=1}^n a_{ij}\sum_{k=1}^n(\mathbf{A}_{\boldsymbol{\kappa}}^{-1})_{jk} \ln(c_k(c))\right)=c_i(c), \;\; 1 \leq i \leq n.\end{equation*}
Then we have
\begin{eqnarray}\label{Eq:limitstep1} && \frac{P\left(\prod\limits_{j=1}^m \left(\hat{X}_j^{(J,c)}\right)^{\hat{a}_{ij}^{(J)}}>c_i(c) x, \;1 \leq i \leq n\right)}{\prod_{j=1}^n P(X_j>x^{\kappa_j})}\\
\label{Eq:twocfactors} &=& \frac{P\left(\prod\limits_{j=1}^m \left(\frac{\hat{X}_j^{(J,c)}}{\hat{c}_j(c)}\right)^{\hat{a}_{ij}^{(J)}}>x, \;1 \leq i \leq n\right)}{\prod_{j=1}^n P(X_j>\hat{c}_j(c)x^{\kappa_j})} \cdot \frac{\prod_{j=1}^n P(X_j>\hat{c}_j(c)x^{\kappa_j})}{\prod_{j=1}^n P(X_j>x^{\kappa_j})}.
\end{eqnarray}
As $x \to \infty$, in view of \eqref{Eq:identitykappahat}, the second factor of \eqref{Eq:twocfactors} converges to
\begin{eqnarray}&&\nonumber \prod_{j=1}^n (\hat{c}_j(c))^{-1}=\prod_{j=1}^n\prod_{i=1}^n (c_i(c))^{-(\mathbf{A}_{\boldsymbol{\kappa}}^{-1})_{ji}}=\prod_{i=1}^n (c_i(c))^{-((\mathbf{A}_{\boldsymbol{\kappa}}^{-1})^T\mathbf{1})_i}=\prod_{i=1}^n (c_i(c))^{-\hat{\kappa}_i}\\
\label{Eq:limitofcs} &=&\prod_{i=1}^n\left(c_i^{-\hat{\kappa}_i}c^{\hat{\kappa}_i \sum_{j=n'+1}^m\hat{a}_{ij}^{(J)}}\right)=c^{\sum_{i=1}^n\sum_{j=n'+1}^m \hat{a}_{ij}^{(J)}\hat{\kappa}_i}\prod_{i=1}^n c_i^{-\hat{\kappa}_i}.
\end{eqnarray}

Note that $P^{\hat{X}_j^{(J,c)}}=P^{X_j}$ for $1 \leq j \leq n$ and \begin{equation*}\frac{\hat{X}_j^{(J,c)}}{\hat{c}_j(c)}=\hat{X}_j^{(J,c)}\geq 1 \mbox{ a.s.}, \;\;  E\left(\frac{\hat{X}_j^{(J,c)}}{\hat{c}_j(c)}\right)<\infty \mbox{ for } n<j\leq m.\end{equation*}
Hence, we can apply Proposition \ref{Pr:mainprop} to see that the first factor of \eqref{Eq:twocfactors} converges to
\begin{equation*} \int_{M(\hat{\mathbf{A}}^{(J)})} \prod_{j=1}^nx_j^{-2}\lebesgue(\mbox{d}(x_j)_{1 \leq j \leq n}) \otimes P^{(\hat{X}_j^{(J,c)})_{n< j \leq m}}(\mbox{d}(x_j)_{n< j \leq m})\end{equation*}
Write $\mathbf{x}_1=(x_1, \ldots, x_n), \mathbf{x}_2=(x_{n+1}, \ldots, x_m)$ for abbreviation and use the substition $\mathbf{y}=\ln(\mathbf{x}_1):=(\ln(x_1), \ldots, \ln(x_n))$ to see that the above expression equals
\begin{align}
\nonumber & \int\limits_{[0,\infty)^{m-n}}\int\limits_{\left\{\mathbf{y} \in \mathbb{R}^n:\exp(\mathbf{A}_{\boldsymbol{\kappa}}\mathbf{y})>(\prod_{j=n+1}^m x_j^{-\hat{a}^{(J)}_{ij}})_{1 \leq i \leq n}\right\}} \hspace{-1cm} \exp(-\sum_{i=1}^n y_i)\lebesgue(\mbox{d}\mathbf{y})P^{(\hat{X}_j^{(J,c)})_{n< j \leq m}}(\mbox{d}\mathbf{x}_2)\\
\nonumber &= \int\limits_{[0,\infty)^{m-n}}\int\limits_{\left\{\mathbf{z} \in \mathbb{R}^n:\exp(\mathbf{z})>(\prod_{j=n+1}^m x_j^{-\hat{a}^{(J)}_{ij}})_{1 \leq i \leq n}\right\}}|\det(\mathbf{A}_{\boldsymbol{\kappa}})|^{-1} \\
\nonumber & \hspace{1cm}\exp\left(-\sum_{k=1}^n\sum_{l=1}^n
(\mathbf{A}_{\boldsymbol{\kappa}}^{-1})_{kl}z_l\right)\lebesgue(\mbox{d}\mathbf{z})\, P^{(\hat{X}_j^{(J,c)})_{n< j \leq m}}(\mbox{d}\mathbf{x}_2) \\
\nonumber &= |\det(\mathbf{A}_{\boldsymbol{\kappa}})|^{-1}\int\limits_{[0,\infty)^{m-n}}\prod_{l=1}^n \int\limits_{(\ln(\prod_{j=n+1}^m x_j^{-\hat{a}^{(J)}_{lj}}),\infty)} \hspace{-1cm} \exp\left(-z_l\sum_{k=1}^n (\mathbf{A}_{\boldsymbol{\kappa}}^{-1})_{kl}\right)\mbox{d}z_l P^{(\hat{X}^{(J,c)}_j)_{n< j \leq m}}(\mbox{d}\mathbf{x}_2) \\
\nonumber &= \frac{|\det(\mathbf{A}_{\boldsymbol{\kappa}})|^{-1}}{\prod_{i=1}^n
((\mathbf{A}_{\boldsymbol{\kappa}}^{-1})^T\mathbf{1})_i}\int_{[0,\infty)^{m-n}}\prod_{l=1}^n\prod_{j=n+1}^m x_j^{\hat{a}_{lj}^{(J)}\sum_{k=1}^n (\mathbf{A}_{\boldsymbol{\kappa}}^{-1})_{kl}}P^{(\hat{X}^{(J,c)}_j)_{n< j \leq m}}(\mbox{d}\mathbf{x}_2) \\
\nonumber &= \frac{|\det(\mathbf{A}_{\boldsymbol{\kappa}})|^{-1}}{\prod_{i=1}^n \hat{\kappa}_i}\prod_{j=n+1}^m E\left(\left(\hat{X}_j^{(J,c)}\right)^{\sum_{l=1}^n \hat{a}^{(J)}_{lj} ((\mathbf{A}_{\boldsymbol{\kappa}}^{-1})^T\mathbf{1})_l}\right)\\
\label{Eq:finalsecondfactor} &= \frac{|\det(\mathbf{A}_{\boldsymbol{\kappa}})|^{-1}}{\prod_{i=1}^n \hat{\kappa}_i}\prod_{j=n'+1}^m \left( E\left(\left(\hat{X}_j^{(J)}\right)^{\sum_{i=1}^n \hat{a}^{(J)}_{ij} \hat{\kappa}_i}\Big|\hat{X}_j^{(J)}\geq c \right)c^{-\sum_{i=1}^n \hat{a}^{(J)}_{ij} \hat{\kappa}_i}\right),
\end{align}
where we used in the final step that $\sum_{i=1}^n \hat{a}^{(J)}_{ij} \hat{\kappa}_i=(\mathbf{1}^T\mathbf{A}_{\boldsymbol{\kappa}}^{-1}\hat{\mathbf{A}}^{(J)})_j=0$ for $n<j \leq n'$, cf.\ \eqref{Eq:alldualcases}. Combine \eqref{Eq:limitofcs} and \eqref{Eq:finalsecondfactor} to see that the expression in  \eqref{Eq:twocfactors} converges to
\begin{equation}\label{Eq:limitforstep1} \frac{|\det(\mathbf{A}_{\boldsymbol{\kappa}})|^{-1}\prod_{i=1}^n c_i^{-\hat{\kappa}_i}}{\prod_{i=1}^n \hat{\kappa}_i}\prod_{j=n'+1}^m E\left(\left(\hat{X}_j^{(J)}\right)^{\sum_{i=1}^n \hat{a}^{(J)}_{ij} \hat{\kappa}_i}\Big|\hat{X}_j^{(J)}\geq c \right)
\end{equation}
as $x \to \infty$. Now, \eqref{Eq:limitforstep1} together with \eqref{Eq:limitctozero} and \eqref{Eq:limitwithcs} yields that
\begin{eqnarray*}
&& \lim_{x \to \infty} \frac{P\left(\prod\limits_{j=1}^m \left(\hat{X}_j^{(J)}\right)^{\hat{a}^{(J)}_{ij}}>c_i x, \; 1 \leq i \leq n\right)}{\prod_{j=1}^nP(X_j>x^{\kappa_j})} \\
&=& \lim_{c \searrow 0}\frac{|\det(\mathbf{A}_{\boldsymbol{\kappa}})|^{-1}\prod_{i=1}^n c_i^{-\hat{\kappa}_i}}{\prod_{i=1}^n \hat{\kappa}_i}\prod_{j=n'+1}^m E\left(\left(\hat{X}_j^{(J)}\right)^{\sum_{i=1}^n \hat{a}^{(J)}_{ij} \hat{\kappa}_i}\mathds{1}_{\{\hat{X}_j^{(J)}\geq c\}}\right)\\
&=& \frac{|\det(\mathbf{A}_{\boldsymbol{\kappa}})|^{-1}\prod_{i=1}^n c_i^{-\hat{\kappa}_i}}{\prod_{i=1}^n \hat{\kappa}_i}\prod_{j=n'+1}^m E\left(\left(\hat{X}_j^{(J)}\right)^{\sum_{i=1}^n \hat{a}^{(J)}_{ij} \hat{\kappa}_i}\right)\\
&=& \frac{|\det(\mathbf{A}_{\boldsymbol{\kappa}})|^{-1}\prod_{i=1}^n c_i^{-\hat{\kappa}_i}}{\prod_{i=1}^n \hat{\kappa}_i}\prod_{j=n'+1}^m E\left(X_j^{\sum_{i=1}^n a_{ij} \hat{\kappa}_i}\right),
\end{eqnarray*}
where we used $\sum_{i=1}^n \hat{a}^{(J)}_{ij} \hat{\kappa}_i=((\hat{\mathbf{A}}^{(J)})^T
(\mathbf{A}_{\boldsymbol{\kappa}}^{-1})^T\mathbf{1})_j=(\mathbf{1}^T
\mathbf{A}_{\boldsymbol{\kappa}}^{-1}\hat{\mathbf{A}}^{(J)})_j>0, n'<j \leq m$ (cf.\ \eqref{Eq:alldualcases}) in the penultimate equality and \eqref{Eq:DefXtilde}, \eqref{Eq:DefAtilde}, \eqref{Eq:DefXJtilde} and \eqref{Eq:DefAJtilde} in the final equality. This expression no longer depends on $J \subset \{n+1, \ldots n'\}$ and therefore \eqref{Eq:firsttransfo}, \eqref{Eq:condsummands} and \eqref{Eq:probwithX^{(J)}} lead to
\begin{eqnarray*}
 && \lim_{x \to \infty} \frac{P\left(\prod_{j=1}^m X_j^{a_{ij}}>c_i x, \; 1 \leq i \leq n\right)}{\prod_{j=1}^nP(X_j>x^{\kappa_j})}=\lim_{x \to \infty} \frac{P\left(\prod_{j=1}^m \hat{X}_j^{\hat{a}_{ij}}>c_i x, \; 1 \leq i \leq n\right)}{\prod_{j=1}^nP(X_j>x^{\kappa_j})}\\
 &=& \frac{|\det(\mathbf{A}_{\boldsymbol{\kappa}})|^{-1}\prod_{i=1}^n c_i^{-\hat{\kappa}_i}}{\prod_{i=1}^n \hat{\kappa}_i}\prod_{j=n'+1}^m E\left(X_j^{\sum_{i=1}^n a_{ij} \hat{\kappa}_i}\right) \\
 &=& |\det \mathbf{A}_{\boldsymbol{\kappa}}|^{-1}\frac{\prod_{i=1}^n c_i^{-(\mathbf{1}^T\mathbf{A}_{\boldsymbol{\kappa}}^{-1})_i}}{\prod_{i=1}^n (\mathbf{1}^T\mathbf{A}_{\boldsymbol{\kappa}}^{-1})_i}\prod_{j:\kappa_j=0} E\left(X_j^{ (\mathbf{1}^T\mathbf{A}_{\boldsymbol{\kappa}}^{-1}\mathbf{A})_j}\right),
\end{eqnarray*}
and so the limit in \eqref{Eq:Maintheorem1} equals the expression in \eqref{Eq:Maintheorem2}. By Theorems~2.4 and 2.5 of \cite{LiReRo14}, this shows that $c(x) P\left((x^{-1}\prod_{j=1}^m X_j^{a_{ij}})_{1 \leq i \leq n} \in \cdot \right)$, $ x>0$, with $c(x)=(\prod_{j=1}^nP(X_j>x^{\kappa_j}))^{-1}$, is relatively compact in $\mathbb{M}_{(0,\infty)^n}$. Furthermore, all accumulation points of this family agree on a generating $\pi$-system. Thus, $P^{(\prod_{j=1}^m X_j^{a_{ij}})_{1 \leq i \leq n}}$ is regularly varying on $(0,\infty)^n$ w.r.t.\ scalar multiplication, cf.\ Example~\ref{ex:LT}. The index of regular variation follows from Lemma and Definition~\ref{LemDef:index} since $c$ is regularly varying with index $-\sum_{j=1}^n\kappa_j=-\sum_{j=1}^m\kappa_j=-\mathbf{1}^T\mathbf{A}_{\boldsymbol{\kappa}}^{-1}\mathbf{1}$, cf.\ \eqref{Eq:optval}.
\end{proof}

\subsection{Auxiliary results}\label{Sec:auxresults}
In the following, we collect two lemmas and a proposition which are needed for the proofs in Sections \ref{Sec:boundsproof} and \ref{Sec:proof}.
\begin{lemma}\label{Lem:makepositive} Let $\boldsymbol{\kappa}=(\kappa_1, \ldots, \kappa_m)^T$ be an optimal solution to \eqref{Eq:LOP}.
\begin{itemize}
\item[(a)] There exists a matrix $\tilde{\mathbf{A}}=(\tilde{a}_{ij}) \in \mathbb{R}^{n \times m}$ such that
\begin{itemize}
\item the columns $j$ in $\tilde{\mathbf{A}}$ for which $\kappa_j>0$ have all positive entries,
\item $\boldsymbol{\kappa}$ is an optimal solution to the linear program
\begin{equation}\label{Eq:posLP}  \mbox{find } \mathbf{x} \geq \mathbf{0} \; \mbox{such that } \tilde{\mathbf{A}} \mathbf{x} \geq \mathbf{1}, \;\;\; \sum_{i=1}^mx_i \to \min!
\end{equation}
\item for all $x, x_1, x_2, \ldots, x_m\geq 0$,
\begin{equation}
\label{Eq:Upperprobbound}\prod_{j=1}^m x_j^{a_{ij}}>x, \;1 \leq i \leq n \; \Rightarrow  \; \prod_{j=1}^m x_j^{\tilde{a}_{ij}}>x, \;1 \leq i \leq n.
\end{equation}
\end{itemize}
\item[(b)] Moreover, if the assumptions of Proposition \ref{Pr:mainprop} hold, then the matrix $\tilde{\mathbf{A}}$ can be chosen such that additionally
\begin{itemize}
\item $\boldsymbol{\kappa}$ is the \emph{unique} optimal solution to the linear program \eqref{Eq:posLP},
\item $\tilde{\mathbf{A}}\boldsymbol{\kappa}=\mathbf{1}.$
\end{itemize}
\item[(c)] If the assumptions of Theorem \ref{Th:maintheorem} hold, then there exists a matrix $\tilde{\mathbf{A}}=(\tilde{a}_{ij}) \in \mathbb{R}^{n \times m}$ such that
\begin{itemize}
\item the columns $j$ in $\tilde{\mathbf{A}}$ for which $(\mathbf{1}^T\mathbf{A}_{\boldsymbol{\kappa}}^{-1}\mathbf{A})_j>0$ have all positive entries,
\item $\boldsymbol{\kappa}$ is the unique optimal solution to the linear program \eqref{Eq:posLP},
\item $\tilde{\mathbf{A}}\boldsymbol{\kappa}=\mathbf{1}$,
\item for all $x, x_1, x_2, \ldots, x_m\geq 0$ \eqref{Eq:Upperprobbound} holds.
\end{itemize}
\end{itemize}
\end{lemma}
\begin{proof} First note that if $a_{ij}>0$ for all $1 \leq i \leq n$ and all $j$ such that $\kappa_j>0$ (cases (a) and (b)) or $(\mathbf{1}^T\mathbf{A}^{-1}_{\boldsymbol{\kappa}}\mathbf{A})_j>0$ (case (c)), then we may simply set $\tilde{\mathbf{A}}=\mathbf{A}$. So, assume the contrary in the following. Set $J:=\{j \in \{1, \ldots, m\}:\kappa_j>0\}$.
Since we have assumed an optimal solution $\boldsymbol{\kappa}$ to \eqref{Eq:LOP}, there also exists an optimal (not necessarily unique) solution $\hat{\boldsymbol{\kappa}}=(\hat{\kappa}_1, \ldots, \hat{\kappa}_n)^T$ to the dual problem \eqref{Eq:DualLO} and this solution satisfies $\sum_{i=1}^n\hat{\kappa}_i=\sum_{j=1}^m\kappa_j$, cf.\ Theorem 2.2 in \cite{Si96}. Furthermore, by the Complementary Slackness Theorem (cf.\ \cite{Si96}, Theorem 2.4) we have $(\mathbf{A}^T\hat{\boldsymbol{\kappa}})_j=1$ for all $j \in J$.
For assertions (a) and (b) let $a_{\min}:=-(\min_{1\leq i \leq n, j \in J} a_{ij})+\epsilon$ for some $\epsilon>0$. By our assumptions, $a_{\min}$ is positive. Define
\begin{equation*} \tilde{\mathbf{A}}=(\tilde{a}_{ij})_{1\leq i \leq n, 1 \leq j \leq m} \;\; \mbox{with} \;\; \tilde{a}_{ij}=\frac{a_{ij}+a_{\min}\sum_{k=1}^n a_{kj}\hat{\kappa}_k}{1+a_{\min}\sum_{k=1}^m\kappa_k}.\end{equation*}
As seen above, we have $\sum_{k=1}^na_{kj}\hat{\kappa}_k=1$ and thus $\tilde{a}_{ij}>0$ for $j \in J$ and all $1\leq i\leq n$.

Note that
\begin{eqnarray}
\nonumber \tilde{\mathbf{A}}\boldsymbol{\kappa}&=& \left(1+a_{\min}\sum_{i=1}^m\kappa_i\right)^{-1}\left((a_{ij}+a_{\min})_{1 \leq i \leq n, j \in J}\right)((\kappa_j)_{j \in J})^T \\
\label{Eq:kappafeasiblefortilde} &\geq & \left(1+a_{\min}\sum_{i=1}^m\kappa_i\right)^{-1}\left(\mathbf{1}+a_{\min}\sum_{i=1}^m\kappa_i \mathbf{1}\right)=\mathbf{1},
\end{eqnarray}
so $\boldsymbol{\kappa}$ is a feasible solution to \eqref{Eq:posLP}. Furthermore, if there would exist a $\boldsymbol{\kappa}'\geq \mathbf{0}$ with $\tilde{\mathbf{A}}\boldsymbol{\kappa}'\geq \mathbf{1}$ and $\sum_{i=1}^m\kappa_i'<\sum_{i=1}^m\kappa_i$, then
\begin{equation}\label{Eq:implicationkappaprime} \sum_{j=1}^m\left(a_{ij}+a_{\min}\sum_{k=1}^na_{kj}\hat{\kappa}_k\right)\kappa_j'  \geq  1+a_{\min}\sum_{k=1}^m\kappa_k, \;\;\; 1 \leq i \leq n,
\end{equation}
and thus
\begin{eqnarray}
\nonumber \sum_{j=1}^m a_{ij}\kappa_j' &\geq & 1+a_{\min}\sum_{k=1}^m\kappa_k-a_{\min}\sum_{k=1}^n\sum_{j=1}^m a_{kj}\hat{\kappa}_k\kappa_j' \\
\nonumber &\geq & 1+a_{\min}\left(\sum_{k=1}^m\kappa_k-\sum_{j=1}^m \kappa_j'\right) \\
\label{Eq:kappaprimealsofeasible} &\geq& 1, \;\;\; 1 \leq i \leq n,
\end{eqnarray}
where we used in the penultimate inequality that $\sum_{k=1}^n a_{kj}\hat{\kappa}_k\leq 1$ and $\kappa_j'\geq 0, 1 \leq j \leq m$. But this implies that $\boldsymbol{\kappa}'$ with $\sum_{i=1}^m\kappa_i'<\sum_{i=1}^m\kappa_i$ would also be a feasible solution to \eqref{Eq:LOP}, in contrast to the assumption about the optimality of $\boldsymbol{\kappa}$. Thus, $\boldsymbol{\kappa}$ is also an optimal solution to \eqref{Eq:posLP}.

We are thus left to show \eqref{Eq:Upperprobbound} for the proof of (a). For $x, x_1, x_2, \ldots, x_m \geq 0$ such that $\prod_{j=1}^m x_j^{a_{ij}}>x, 1 \leq i \leq n,$ we have
\begin{equation*} \prod_{j=1}^m x_j^{a_{ij}a_{\min}\hat{\kappa}_i}\geq x^{a_{\min}\hat{\kappa}_i},  \;\;\; 1 \leq i \leq n, \end{equation*}
with strict inequality if $\hat{\kappa}_i>0$, which must be the case for at least one $1 \leq i \leq n$. So by multiplication of left hand sides and right hand sides we obtain
\begin{align*}
&&\left(\prod_{j=1}^m x_j^{a_{ij}}\right)\prod_{k=1}^n\prod_{j=1}^m x_j^{a_{kj}a_{\min}\hat{\kappa}_k}&>x \prod_{k=1}^nx^{a_{\min} \hat{\kappa}_k}=x^{1+a_{\min}\sum_{k=1}^n\hat{\kappa}_k}, \;\; 1 \leq i \leq n \\
&\Leftrightarrow & \prod_{j=1}^mx_j^{a_{ij}+a_{\min}\sum_{k=1}^na_{kj}\hat{\kappa}_k}&>x^{1+a_{\min}\sum_{k=1}^n\hat{\kappa}_k}, \;\; 1 \leq i \leq n\\
&\Leftrightarrow & \prod_{j=1}^mx_j^{\frac{a_{ij}+a_{\min}\sum_{k=1}^na_{kj}\hat{\kappa}_k}{1+a_{\min}\sum_{k=1}^n\hat{\kappa}_k}}&>x, \;\; 1 \leq i \leq n\\
&\Leftrightarrow & \prod_{j=1}^mx_j^{\tilde{a}_{ij}}&>x, \;\; 1 \leq i \leq n.
\end{align*}
Thus, \eqref{Eq:Upperprobbound} holds.

For the proof of (b), we use that the additional assumption implies that $\boldsymbol{\kappa}$ is the \emph{unique} optimal solution to \eqref{Eq:LOP} and that $\mathbf{A}\boldsymbol{\kappa}=\mathbf{1}$. Similar to \eqref{Eq:kappafeasiblefortilde} one shows that $\tilde{\mathbf{A}}\boldsymbol{\kappa}=\mathbf{1}$. Furthermore, if there would exist a $\boldsymbol{\kappa}'\neq \boldsymbol{\kappa}$ with $\tilde{\mathbf{A}}\boldsymbol{\kappa}'\geq \mathbf{1}$ and $\sum_{j=1}^m\kappa_j'\leq \sum_{j=1}^m\kappa_j$ then one shows analogously to \eqref{Eq:kappaprimealsofeasible} that this would imply that the optimal solution $\boldsymbol{\kappa}$ to \eqref{Eq:LOP} is not unique. This shows that $\boldsymbol{\kappa}$ is the unique optimal solution to \eqref{Eq:posLP} and proves (b).

For the proof of (c), we use that the additional assumption implies that $\boldsymbol{\kappa}$ is the \emph{unique} optimal solution to \eqref{Eq:LOP} and that $\hat{\boldsymbol{\kappa}}=(\mathbf{A}_{\boldsymbol{\kappa}}^{-1})^T\mathbf{1}$ is the unique solution to \eqref{Eq:DualLO}, cf.\ the beginning of the proof of Theorem \ref{Th:maintheorem}. Let for some $\epsilon>0$
\begin{eqnarray*}a_{\min}^{(c)}&:=&-\min_{1 \leq i \leq n, j: (\mathbf{1}^T\mathbf{A}_{\boldsymbol{\kappa}}^{-1}\mathbf{A})_j>0}\frac{a_{ij}}{ (\mathbf{1}^T\mathbf{A}_{\boldsymbol{\kappa}}^{-1}\mathbf{A})_j}+\epsilon\\
&=&-\min_{1 \leq i \leq n, j: (\mathbf{1}^T\mathbf{A}_{\boldsymbol{\kappa}}^{-1}\mathbf{A})_j>0}\frac{a_{ij}}{ \sum_{k=1}^na_{kj}\hat{\kappa}_k}+\epsilon
\end{eqnarray*}

which is positive by our assumptions. Define
\begin{equation*} \tilde{\mathbf{A}}^{(c)}=(\tilde{a}_{ij}^{(c)})_{1\leq i \leq n, 1 \leq j \leq m} \;\; \mbox{with} \;\; \tilde{a}_{ij}=\frac{a_{ij}+a_{\min}^{(c)}\sum_{k=1}^na_{kj}\hat{\kappa}_k}{1+a_{\min}^{(c)}\sum_{k=1}^m\kappa_k}.\end{equation*}
We have thus $\tilde{a}_{ij}^{(c)}>0$ for those $j$ with $(\mathbf{1}^T\mathbf{A}_{\boldsymbol{\kappa}}^{-1}\mathbf{A})_j>0$ and all $1\leq i\leq n$. The rest of the proof for assertion (c) follows analogously to the proof of (a) and (b) which did not depend on the value of $a_{\min}>0$.
\end{proof}
\begin{lemma}\label{Lem:boundedbysum}
Let the assumptions of Proposition \ref{Pr:mainprop} hold and assume in addition that $a_{ij}>0$ for all $1\leq i \leq n$ and those $1 \leq j \leq m$ for which $\kappa_j>0$. Then for all $j$ with $\kappa_j>0$ there exists $\epsilon>0$ such that
\begin{equation}\label{Eq:boundedbysum}
\min_{1 \leq i \leq n}\sum\limits_{1 \leq k \leq m, k \neq j}\frac{a_{ik}}{a_{ij}}x_k \leq (1-\epsilon)\sum\limits_{1 \leq k \leq m, k \neq j}x_k
\end{equation}
for all $(x_k)_{1 \leq k \leq m, k \neq j} \in [0,\infty)^{m-1}$.
\end{lemma}
\begin{proof} For ease of notation and w.l.o.g., let us assume that $\kappa_1>0$ and treat only the case $j=1$.
For $(x_k)_{2 \leq k \leq m}=\mathbf{0}$ the inequality holds for all $\epsilon>0$. The rest of the proof is by contradiction. Let $(\epsilon_l)_{l \in \mathbb{N}}$ be a sequence such that $\epsilon_l>0$ for all $l \in \mathbb{N}$ and $\epsilon_l \searrow 0$ for $l \to \infty$. Assume that for each $l$ there exists $(x_k^{(l)})_{2 \leq k \leq m} \in [0,\infty)^{m-1}\setminus\{\mathbf{0}\}$ such that
\begin{align}
\nonumber &&\min_{1 \leq i \leq n}\sum_{j=2}^m \frac{a_{ij}}{a_{i1}}x_j^{(l)} &\geq (1-\epsilon_l)\sum_{k=2}^m x_k^{(l)} & \\
\label{Eq:epsilonbound}&\Leftrightarrow & \sum_{j=2}^m a_{ij}\frac{x_j^{(l)}}{\sum_{k=2}^m x_k^{(l)}} &\geq (1-\epsilon_l)a_{i1}, & 1 \leq i \leq n,
\end{align}
where we used that $a_{i1}>0, 1 \leq i \leq n,$ by our assumption. Define now 
\begin{equation*}\tilde{\boldsymbol{\kappa}}^{(l)}=(\tilde{\kappa}_1^{(l)}, \ldots, \tilde{\kappa}_m^{(l)})^T=\left(0,\kappa_2+\frac{x_2^{(l)}}{\sum_{k=2}^m x_k^{(l)}}\kappa_1, \ldots, \kappa_m+\frac{x_m^{(l)}}{\sum_{k=2}^m x_k^{(l)}}\kappa_1\right)^T\geq \mathbf{0}\end{equation*}
for all $l \in \mathbb{N}$. We have
\begin{equation*} \sum_{j=1}^m\tilde{\kappa}_j^{(l)}=\sum_{j=1}^m \kappa_j \end{equation*}
for all $l \in \mathbb{N}$. Furthermore,
\begin{eqnarray}
\nonumber \sum_{j=1}^m a_{ij}\tilde{\kappa}_j^{(l)}&=& \sum_{j=2}^m a_{ij}\left(\kappa_j+\frac{x_j^{(l)}}{\sum_{k=2}^m x_k^{(l)}}\kappa_1\right)\\
\nonumber &=& \sum_{j=2}^m a_{ij} \kappa_j + \sum_{j=2}^ma_{ij}\frac{x_j^{(l)}}{\sum_{k=2}^m x_k^{(l)}}\kappa_1 \\
\label{Eq:contraepsilon}& \geq & 1-a_{i1}\kappa_1 +(1-\epsilon_l)a_{i1}\kappa_1=1-\epsilon_la_{i1}\kappa_1,
\end{eqnarray}
for all $l \in \mathbb{N}$, where we used $\mathbf{A}\boldsymbol{\kappa}\geq \mathbf{1}$ and \eqref{Eq:epsilonbound} in the last step. For $l \to \infty$, the bounded sequence $\tilde{\boldsymbol{\kappa}}^{(l)}$ must have an accumulation point $\tilde{\boldsymbol{\kappa}} \neq \boldsymbol{\kappa}$ (because $\tilde{\kappa}_1=0<\kappa_1$) and $\tilde{\boldsymbol{\kappa}} \geq \mathbf{0}$. But
\begin{equation*} \sum_{j=1}^m\tilde{\kappa}_j=\sum_{j=1}^m\kappa_j \;\;\; \mbox{and} \;\;\; \sum_{j=1}^ma_{ij}\tilde{\kappa}_j\geq 1 \end{equation*}
by \eqref{Eq:contraepsilon}, so our optimal solution $\boldsymbol{\kappa}$ would not be unique, in contradiction to our assumptions. Thus, for some $\epsilon>0$, the inequality \eqref{Eq:boundedbysum} holds for all $(x_j)_{2 \leq j \leq m} \in [0,\infty)^{m-1}$.
\end{proof}
\begin{proposition}\label{Lem:momentsconverge}
Assume that
\begin{align}\nonumber& \lim_{y \to \infty} \frac{P\left((X_j)_{1 \leq j \leq m} \in y \otimes_{\kappa}M(\mathbf{A})\right)}{\prod\limits_{j: \kappa_j>0}P(X_j>y^{\kappa_j})}\\
\label{Eq:Lemma34i} &=\int\limits_{M(\mathbf{A})}\prod_{j:\kappa_j>0}x_j^{-2}\lebesgue(\mbox{d}(x_j)_{\{j:\kappa_j>0\}}) \otimes P^{(X_j)_{\{j:\kappa_j=0\}}}(\mbox{d}(x_j)_{\{j:\kappa_j=0\}}) \in [0,\infty)
\end{align}
holds for all $X_1, \ldots, X_m$ and all matrices $\mathbf{A}$ which satisfy the assumptions of Proposition \ref{Pr:mainprop}, with $\boldsymbol{\kappa}$ being the unique solution to \eqref{Eq:LOP} and $M(\mathbf{A})$ as in Proposition \ref{Pr:mainprop}. Then also
\begin{align}&\nonumber \lim_{y \to \infty} \frac{\int\limits_{M(\mathbf{A})} \prod_{j=1}^m x_j^{\beta_j} P^{(X_j/y^{\kappa_j})_{\{1 \leq j \leq m\}}}(\mbox{d}\mathbf{x})}{\prod\limits_{j: \kappa_j>0}P(X_j>y^{\kappa_j})}\\
\label{Eq:Lemma34ii}&=\int\limits_{M(\mathbf{A})}\prod_{j=1}^m x_j^{\beta_j} \prod_{j:\kappa_j>0} x_j^{-2}\lebesgue(\mbox{d}(x_j)_{\{j:\kappa_j>0\}}) \otimes P^{(X_j)_{\{j:\kappa_j=0\}}}(\mbox{d}(x_j)_{\{j:\kappa_j=0\}}) \in [0,\infty)
\end{align}
for all $X_1, \ldots, X_m, \mathbf{A}$ and $\boldsymbol{\kappa}$ as above and all $\beta_j \in [0,1), 1 \leq j \leq m$.
\end{proposition}
\begin{remark}\label{Rem:Karamata}
As the preceding proposition is used in the induction step of the proof of Proposition \ref{Pr:mainprop}, the convergence \eqref{Eq:Lemma34i} had to be assumed. However, since Proposition \ref{Pr:mainprop} shows that \eqref{Eq:Lemma34i} holds for all $X_1, \ldots, X_m$ and all matrices $\mathbf{A}$ which satisfy the assumptions, the convergence in \eqref{Eq:Lemma34ii} follows. The result may thus be regarded as a multivariate version of the direct half of Karamata's Theorem.
\end{remark}
\begin{proof}[Proof of Proposition \ref{Lem:momentsconverge}]
Define independent random variables $X'_j, 1 \leq j \leq m,$ such that $X'_j$ has $P^{X_j}$-density $x \mapsto \mathds{1}_{[0,\infty)}(x)x^{\beta_j}(E(X_j^{\beta_j}))^{-1}, 1 \leq j \leq m$. This is possible because all $\beta_j \in [0,1)$ and thus $E(X_j^{\beta_j})<\infty$ by our assumptions. For those $1 \leq j \leq m$ with $\kappa_j>0$ the random variable $X'_j$ is regularly varying with index $-(1-\beta_j)$, because
\begin{equation*} \lim_{x \to \infty}\frac{P(X'_j>x)}{x^{\beta_j}P(X_j>x)}=\lim_{x \to \infty}\frac{\int_x^\infty y^{\beta_j}P^{X_j}(\mbox{d}y)}{E(X^{\beta_j})x^{\beta_j}P(X_j>x)}=(E(X^{\beta_j})(1-\beta_j))^{-1} \end{equation*}
for all $1 \leq j \leq m$ by Karamata's Theorem (cf.\ \cite{Cl83}, Lemma 1.1). Thus, for $1 \leq j \leq m$ with $\kappa_j>0$ the random variable $\tilde{X}_j:=(X_j')^{1-\beta_j}$ is regularly varying with index $-1$ and
\begin{equation}\label{Eq:asymptildeX} \lim_{x \to \infty}\frac{P(\tilde{X}_j>x^{1-\beta_j})}{x^{\beta_j}P(X_j>x)}=(E(X^{\beta_j})(1-\beta_j))^{-1}.
\end{equation}
For $1 \leq j \leq m$ with $\kappa_j=0$ we have $\tilde{X}_j \geq 1$ a.s. because we assumed $X_j \geq 1$ a.s. Furthermore, we have for all $\delta \in (0,1]$ and all $j$ with $\kappa_j=0$ that
\begin{equation*} E(\tilde{X}_j^{1-\delta})=\frac{\int_1^\infty x^{(1-\delta)(1-\beta_j)}x^{\beta_j}P^{X_j}(\mbox{d}x)}{E(X_j^{\beta_j})}=\frac{E(X_j^{1-(1-\beta_j)\delta})}{E(X_j^{\beta_j})}<\infty.\end{equation*}
Thus, the random variables $\tilde{X}_j, 1 \leq j \leq m,$ satisfy the assumptions of Proposition \ref{Pr:mainprop}.
Set now
\begin{equation*} \tilde{\mathbf{A}}=(\tilde{a}_{ij}):=((1-\beta_j)^{-1}a_{ij}) \in \mathbb{R}^{n \times m}.\end{equation*}
Then $\tilde{\boldsymbol{\kappa}}:=((1-\beta_j)\kappa_j)_{1 \leq j \leq m}$ is the unique solution to the linear program
\begin{equation*} \mbox{find } \mathbf{x} \geq \mathbf{0} \; \mbox{such that } \tilde{\mathbf{A}} \mathbf{x} \geq \mathbf{1}, \;\;\; \sum_{i=1}^mx_i \to \min!
\end{equation*}
and $\tilde{\mathbf{A}}\tilde{\boldsymbol{\kappa}}=\mathbf{1}$. Set
\begin{equation*} M(\tilde{\mathbf{A}})=\left\{(x_1, \ldots, x_m): \prod_{j=1}^m x_j^{\tilde{a}_{ij}}>1, 1 \leq i \leq n\right\}.\end{equation*}
Then,
\begin{align*}
& &(\tilde{X}_j)_{1 \leq j \leq m}  \in y &\otimes_{\tilde{\boldsymbol{\kappa}}}M(\tilde{\mathbf{A}})& \\
& \Leftrightarrow & \prod_{j=1}^m \left(\frac{\tilde{X}_j}{y^{\tilde{\kappa}_j}}\right)^{\tilde{a}_{ij}}>1, & \;\;\;1 \leq i \leq n,&\\
& \Leftrightarrow & \prod_{j=1}^m \left(\frac{(X_j')^{1-\beta_j}}{y^{(1-\beta_j)\kappa_j}}\right)^{(1-\beta_j)^{-1}a_{ij}}>1, &\;\;\;1 \leq i \leq n,&\\
& \Leftrightarrow & (X_j')_{1 \leq j \leq m}  \in y &\otimes_{\boldsymbol{\kappa}}M(\mathbf{A}),&
\end{align*}
and so
\begin{eqnarray*}
P\left((\tilde{X}_j)_{1 \leq j \leq m} \in y \otimes_{\tilde{\boldsymbol{\kappa}}}M(\tilde{\mathbf{A}})\right)&=&\int\limits_{y \otimes_{\boldsymbol{\kappa}}M(\mathbf{A})} P^{(X_j')_{1 \leq j \leq m} }(\mbox{d}\mathbf{x})\\
&=&\int\limits_{y \otimes_{\boldsymbol{\kappa}}M(\mathbf{A})} \prod_{j=1}^m \frac{x_j^{\beta_j}}{E(X_j^{\beta_j})}P^{(X_j)_{1 \leq j \leq m} }(\mbox{d}\mathbf{x})\\
&=&\int\limits_{M(\mathbf{A})} \prod_{j=1}^m \frac{(y^{\kappa_j}x_j)^{\beta_j}}{E(X_j^{\beta_j})}P^{(X_j/y^{\kappa_j})_{1 \leq j \leq m} }(\mbox{d}\mathbf{x}).
\end{eqnarray*}
Thus,
\begin{align}
\nonumber & \lim_{y \to \infty} \frac{\int\limits_{M(\mathbf{A})} \prod_{j=1}^m x_j^{\beta_j} P^{(X_j/y^{\kappa_j})_{1 \leq j \leq m}}(\mbox{d}\mathbf{x})}{\prod\limits_{j: \kappa_j>0}P(X_j>y^{\kappa_j})}\\
\label{Eq:twofinalfactors}&= \lim_{y \to \infty} \frac{P\left((\tilde{X}_j)_{1 \leq j \leq m} \in y \otimes_{\tilde{\boldsymbol{\kappa}}}M(\tilde{\mathbf{A}})\right)}{\prod\limits_{j: \kappa_j>0}P(\tilde{X_j}>y^{\tilde{\kappa}_j})}\cdot \frac{\prod\limits_{j=1}^mE(X_j^{\beta_j})\prod\limits_{j: \kappa_j>0}P(\tilde{X_j}>y^{\tilde{\kappa}_j})}{\prod\limits_{j=1}^my^{\kappa_j\beta_j}\prod\limits_{j: \kappa_j>0}P(X_j>y^{\kappa_j})},
\end{align}
where the first factor converges to
\begin{equation}\label{Eq:limitAtilde} \int\limits_{M(\tilde{\mathbf{A}})}\prod_{j:\kappa_j>0}x_j^{-2}\lebesgue(\mbox{d}(x_j)_{\{j:\kappa_j>0\}}) \otimes P^{(\tilde{X}_j)_{\{j:\kappa_j=0\}}}(\mbox{d}(x_j)_{\{j:\kappa_j=0\}})
\end{equation}
by the assumption. Substitute $(y_1, \ldots, y_m):=(x_1^{(1-\beta_1)^{-1}}, \ldots, x_m^{(1-\beta_m)^{-1}})$ and note that $(x_1, \ldots, x_m) \in M(\tilde{\mathbf{A}})$ is equivalent to $(y_1, \ldots, y_m) \in M(\mathbf{A})$, so the expression in \eqref{Eq:limitAtilde} equals
\begin{eqnarray*}
&&\int\limits_{M(\mathbf{A})}\prod_{j:\kappa_j>0}(1-\beta_j)y_j^{\beta_j-2}\lebesgue(\mbox{d}(y_j)_{\{j:\kappa_j>0\}}) \otimes P^{(X_j')_{\{j:\kappa_j=0\}}}(\mbox{d}(y_j)_{\{j:\kappa_j=0\}}) \\
&=&\int\limits_{M(\mathbf{A})}\left(\prod_{j:\kappa_j>0}(1-\beta_j)y_j^{\beta_j-2}\right)\left(\prod_{j:\kappa_j=0}E(X_j^{\beta_j})^{-1}y_j^{\beta_j}\right)\\
&& \hspace{3cm} \lebesgue(\mbox{d}(y_j)_{\{j:\kappa_j>0\}}) \otimes P^{(X_j)_{\{j:\kappa_j=0\}}}(\mbox{d}(y_j)_{\{j:\kappa_j=0\}}).
\end{eqnarray*}
The second factor in \eqref{Eq:twofinalfactors} converges to
\begin{equation*}\frac{\prod_{j:\kappa_j=0}E(X_j^{\beta_j})}{\prod_{j:\kappa_j>0}(1-\beta_j)}\end{equation*} 
by \eqref{Eq:asymptildeX}. Taken together, this yields the statement of the proposition.
\end{proof}
\section*{Acknowledgement}
We thank an anonymous referee for useful comments which helped to improve an earlier version of this paper. This work was supported by the German Research Foundation DFG, grant no JA 2160/1.





\end{document}